\documentclass[12pt,a4paper,oneside]{amsart}
\usepackage{amsmath,amssymb,amsthm,hyperref}
\usepackage{enumerate}
\usepackage{amsfonts}
\usepackage{amsthm}
\usepackage{amsmath}
\usepackage{amssymb}
\usepackage{cancel}
\usepackage{t1enc}
\usepackage{graphicx}
\usepackage{mathrsfs}
\usepackage{mathptmx}
\usepackage{lipsum}
\usepackage{tikz}
\usepackage{version} 
\usepackage{color}
\usepackage{enumerate}
\usepackage{datetime}
\numberwithin{equation}{section}
\usepackage[bf]{caption}     
\usepackage{geometry}        
\usepackage{bm}  
\usepackage{buczosierlyukas}
\usepackage{mathtools}
\usepackage[utf8]{inputenc}
\usepackage[normalem]{ulem}

\newtheorem{theorem}{Theorem}[section]

\newtheorem{lemma}[theorem]{Lemma}

\DeclareMathAlphabet{\mathpzc}{OT1}{pzc}{m}{it}
{\theoremstyle{definition}
\newtheorem{definition}[theorem]{Definition}
\newtheorem{remark}[theorem]{Remark}
\newtheorem{notation}[theorem]{Notation}

\newtheorem{claim}[theorem]{Claim}

}

\newcommand{\conv}{\mathrm{conv}}

\definecolor{aquam}{rgb}{0.5,1.0,1.0}
\definecolor{bbrown}{rgb}{0.75,0.38,0.15}
\definecolor{Cyan}{rgb}{0,0.6,0.6}
\definecolor{Darkblue}{rgb}{0,0,1}
\definecolor{Dodgerblue2}{rgb}{0,0.5,1}
\definecolor{Green}{rgb}{0,0.6,0.06}
\definecolor{Kahki}{rgb}{1,1,0.5}
\definecolor{Magenta}{rgb}{0.7,0,0.7}
\definecolor{bMagenta}{rgb}{1,.6,1}
\definecolor{Orange}{rgb}{0.8,0.3,0}
\definecolor{dOrchid}{rgb}{0.7,0.2,0.4}
\definecolor{Orchid}{rgb}{1,0.5,1}
\definecolor{Purple}{rgb}{0.65,0.07,0.85}
\definecolor{Royalblue}{rgb}{0.6,0.85,0.87}
\definecolor{Tan}{rgb}{0.54,0.42,0.23}
\definecolor{bTan}{rgb}{0.94,0.82,0.63}
\definecolor{zoltan}{rgb}{0,0.1,0.3}
\definecolor{Turquoise}{rgb}{0,0.85,0.87}
\definecolor{Yellow}{rgb}{1,1,0}
\definecolor{darkamber}{rgb}{0.4,0.19,0.28}
\definecolor{bYellow}{rgb}{1,1,0.6}
\definecolor{bRed}{rgb}{1,0.7,0.7}
\definecolor{boxcolb}{rgb}{0.87,0.77,0.75}
\definecolor{boxcol}{rgb}{0.6,0.85,0.87}
\definecolor{boxcolgreen}{rgb}{0.64,0.93,0.79}
\definecolor{boxcolaa}{rgb}{.75,.99,.70}
\definecolor{boxcolbb}{rgb}{0.39,0.50,0.56}
\definecolor{boxcolcc}{rgb}{1,0.81,0.65}
\definecolor{yy}{rgb}{0.43,0.21,.18}
\definecolor{gA}{gray}{0.5}
\definecolor{gB}{gray}{0.8}
\definecolor{gC}{gray}{0.9}

\usepackage{xcolor}

\hypersetup{
  colorlinks   = true, 
  urlcolor     = blue, 
  linkcolor    = blue, 
  citecolor   = red 
}

\begin{document}

\title{Better estimates of  H\"older thickness of fractals}
\author{Zolt\'an Buczolich$^*$}
\address[Z. Buczolich, B. Maga]{Department of Analysis, ELTE E\"otv\"os Lor\'and University, 
P\'azm\'any P\'eter S\'et\'any 1/c, 1117 Budapest, Hungary}
\address[B. Maga, G. Vértesy]{HUN-REN Alfr\'ed R\'enyi Institute of Mathematics, Re\'altanoda street 13-15, 1053 Budapest, Hungary}
\email{zoltan.buczolich@ttk.elte.hu}
\urladdr{http://buczo.web.elte.hu, ORCID Id: 0000-0001-5481-8797}

\author{Bal\'azs Maga$^\text{\textdagger}$}
\email{mbalazs0701@gmail.com}
\urladdr{  http://magab.web.elte.hu/}

\author{ G\'asp\'ar V\'ertesy$^\text{\textdaggerdbl}$}
\email{vertesy.gaspar@gmail.com}

\thanks{\scriptsize $^*$
 This author was also supported by the Hungarian National Research, Development and Innovation Office--NKFIH, Grant 124003.
}
\thanks{\scriptsize $^\text{\textdagger}$ This author was supported by the \'UNKP-23-4 New National Excellence Program of the Ministry for Culture and
Innovation from the source of the National Research, Development and Innovation Fund, 
and by the Hungarian National Research, Development and Innovation Office-NKFIH, Grant 124749.}
\thanks{\scriptsize $^\text{\textdaggerdbl}$ This author was supported by the \'UNKP-20-3 New National Excellence Program of the Ministry for Innovation and Technology from the source of the National Research, Development and Innovation Fund, and by the Hungarian National Research, Development and Innovation Office–NKFIH, Grant 124749.  
}

\date{\today}

\begin{abstract}
    Dimensions of level sets of generic continuous functions and generic Hölder functions defined on a fractal $F$ encode information about
    the geometry, ``the thickness" of $F$. While in the continuous case this quantity is related to a reasonably tame dimension notion which is called the topological Hausdorff dimension of $F$,
    the Hölder case seems to be highly nontrivial. A number of earlier papers attempted to deal with this problem, carrying out investigation
    in the case of Hausdorff dimension and box dimension.
    In this paper we continue our study of the Hausdorff dimension of almost every level set of generic $1$-H\"older-$\aaa$ functions,
    denoted by $D_{*}(\aaa, F)$.
    We substantially improve
    previous lower and upper bounds on $D_{*}(\aaa, \Delta)$, where $\Delta$ is the Sierpiński triangle, achieving asymptotically equal bounds as $\alpha\to 0+$.
    Using a similar argument, we also give an even stronger lower bound on the generic lower box dimension of level sets.
    Finally, we construct a connected fractal $F$ on which there is a 
    phase transition of $D_{*}(\aaa, F)$, thus providing the first example exhibiting this behaviour.
\end{abstract}

\maketitle

\setcounter{tocdepth}{3}

\tableofcontents


\section{Introduction}

The Hausdorff dimension of level sets of the generic continuous function defined on $[0, 1]^p$ was determined in \cite{BK}. By introducing the concept of topological Hausdorff dimension,
\cite{BBEtoph} generalized this result to a large class of fractals. The connection of this new notion of dimension and the geometry of level sets was investigated more thoroughly 
in \cite{BBElevel}. We note that topological Hausdorff dimension
sparked interest on its own right: in \cite{MaZha}  the authors studied topological Hausdorff dimension of fractal squares, and it was also mentioned and used in Physics papers, see for example
\cite{Balankintoph}, \cite{Balankinfracspace}, \cite{Balankintransport}, \cite{Balankintoph2}, and  \cite{Balankinfluid}.

One might find studying the level sets of continuous functions slightly inconvenient, as such functions depend only on the topology of the domain $F$, being a subset of some Euclidean space in what follows. 
While this independence from the metric structure is not inherited by the Hausdorff dimension of the level sets, due to Hausdorff dimension's dependence on the metric,
it seems reasonable to expect that the investigation of other classes of functions, subject to
certain metric conditions can be more fruitful. 
The simplest and most natural way to impose such a condition is considering H\"older functions instead of arbitrary continuous functions.
More specifically, we investigate the level sets of generic 1-H\"older-$\aaa$ functions, where genericity is understood in the Baire category sense in the topology given by the supremum norm.

We started to deal with this question in \cite{sier}. We defined $D_{*}(\aaa, F)$ as the essential supremum 
of the Hausdorff dimensions of the level sets of a generic
$1$-H\"older-$\aaa$ function defined on $F$, and verified that this quantity is well-defined for compact $F$.
We also showed that if $\alpha\in(0,1)$, then for connected self{-}{}similar sets, like the Sierpi\'nski triangle
$D_{*}(\aaa, F)$ equals  the Hausdorff dimension 
of almost every level-set in the range of  a generic 1-H\"older-$\aaa$ function, that is, 
one can talk about the Lebesgue-typical dimension of the generic 1-H\"older-$\aaa$ function.
We extended the notation $D_{*}(\aaa, F)$ to $\aaa = 0$ by taking the Lebesgue-typical dimension of the generic continuous function.

In a recently published preprint \cite{buczolich2023box}, the first two authors investigated the box dimension of level sets from a similar point of view,
by introducing a framework akin to the one just described. We introduced the quantities $D_{\underline{B}*}(\aaa, F)$ and $D_{\overline{B}*}(\aaa, F)$, analogous to $D_{*}(\aaa, F)$,
which in some sense correspond to the generic lower and upper box dimension of level sets.

It turned out that the Hausdorff dimension and the upper box dimension of level sets are dual in a certain sense:
while the generic value of the former one measures how well the level sets can be compressed inside the fractal, the latter one concerns with
how well they can be spread out. The problem of lower box dimension of level sets remained mostly open: while the generic Hausdorff dimension of level sets obviously cannot exceed the 
generic lower box dimension of level sets, that is $D_{*}(\aaa, F)\leq D_{\underline{B}*}(\aaa, F)$, we argued that it seems difficult to actually prove stronger upper bounds for the former one, 
as currently we have to rely on the same toolset.
(For details, see discussion surrounding \cite[Theorems~3.2-3.3]{buczolich2023box}.) 

In \cite{sierc}, we carried on our research about $D_{*}(\aaa, F)$ and presented results of more quantitative nature. 
In particular, with $\Delta$ being the Sierpiński triangle, lower and upper estimates on $D_{*}(\aaa, \Delta)$ were provided. One of the main goals of this paper is to improve these bounds.
Capitalizing on new ideas, in Sections \ref{sec:sier_triangle_lower} and \ref{sec:sier_triangle_upper}, we prove 
substantially stronger bounds from both sides. On Figure \ref{fig:siertestimates} one can see the improvement: actually, the new bounds are asymptotically equal as $\alpha\to 0+$.

\begin{figure}[h] 
	\includegraphics[scale=0.75]{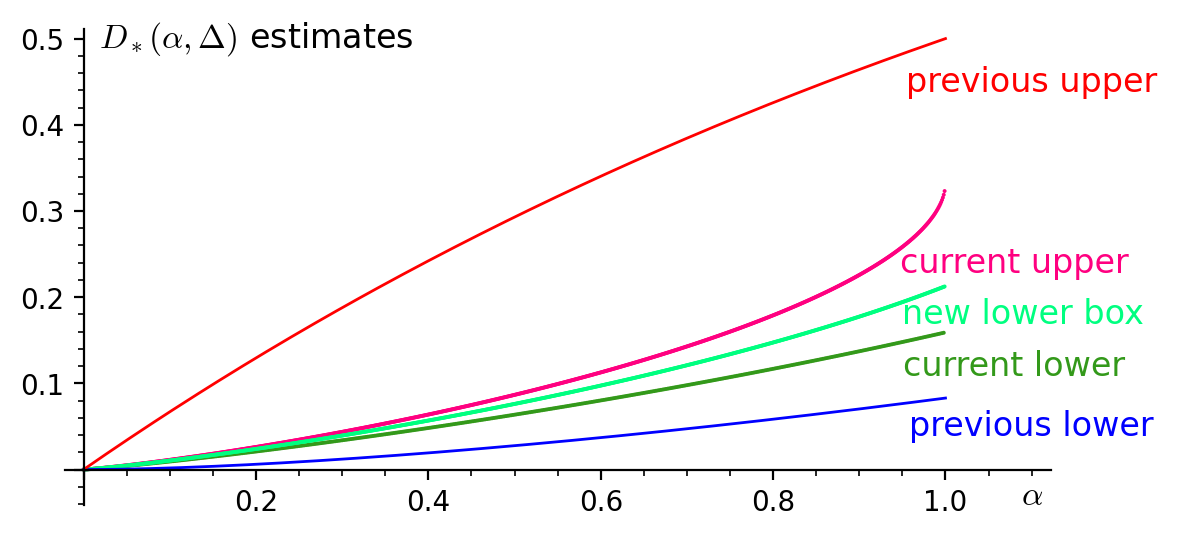}

    \caption{Earlier and current estimates on $D_*(\aaa, \Delta)$, new lower box standing for the lower bound on $D_{\underline{B}*}(\aaa, \Delta)$} 
   
     \label{fig:siertestimates}
\end{figure}

In this paper, we also take an important first step for distinguishing the quantities $D_{*}(\aaa, F), D_{\underline{B}*}(\aaa, F)$, which seems to be a delicate problem, as noted above. 
We conclude Section \ref{sec:sier_triangle_lower} by presenting a lower bound on
$D_{\underline{B}*}(\aaa, \Delta)$ (also displayed on Figure \ref{fig:siertestimates}), which relies on the same approach as the preceding proof of our lower bound on $D_{*}(\aaa, \Delta)$, 
but exceeds it for any $\alpha>0$.
While this development does not yield that 
$D_{*}(\aaa, \Delta)<D_{\underline{B}*}(\aaa, F)$ for any $\alpha$, it certainly fuels this hope. It should be noted that the upper bound $D_{*}(\aaa, \Delta)$ we provide is also an upper bound on 
$D_{\underline{B}*}(\aaa, F)$, thus $D_{*}(\aaa, \Delta)$ and $D_{\underline{B}*}(\aaa, F)$ are asymptotically equal as $\alpha\to 0+$, further emphasizing that they cannot differ by much, while they
sharply differ from the analogous quantity defined for the upper box dimension.

In \cite{sierc}, the notion of phase transition was also introduced: $D_*(\cdot, F)$ has phase transition, if for some
$\alpha_0>0$, we have $D_*(\alpha, F) = D_*(0, F)$ for $\alpha<\alpha_0$, but $D_*(\alpha, F) > D_*(0, F)$ for $\alpha>\alpha_0$.
It is natural to ask whether such an $F$ exists, and in \cite{sierc}, we answered this question affirmatively.
However, the $F$ provided was not connected, which illustrated that our understanding of phase transition
there was incomplete: disconnected fractals easily exhibit peculiar behaviour
due to the fact that disconnectedness enables one to define functions with less limitation. It remained a nontrivial question whether a connected example for phase transition can be produced,
one that sparked interest during seminars when we presented our results.
In Section \ref{sec:phase_transition}, we answer this question affirmatively as well by providing a connected self-similar example. We note that our construction exhibits phase transition on $D_{\underline{B}*}(\aaa, F)$ as well.


\section{Notation and preliminaries} \label{sec:prelimin}

The interior of a set $H$ is denoted by $\inte{H}$.

Following \cite{sier}, we introduce the following notation: if $(M,\mathpzc{d})$ 
is a metric space, then let $C(M)$ be the space of continuous functions  from $M$ to $\R$.   
Moreover for every $\alpha\in(0,1]$ and $c>0$ set
$$
C_{c}^{\alpha}(M) := \{ f\in C(M) : \text{for every $a,b\in M$ we have $|f(a)-f(b)|\le c(\mathpzc{d}(a,b))^\alpha$}\}
$$
and 
$$
C_{c^-}^{\alpha}(M) := \Union_{c'\in [0,c)} C_{c'}^{\alpha}(M).
$$

In what follows, our domain is some compact $F\subseteq \mathbb{R}^p$. Let $D^f(r,F)=D^f(r)=\dim_H(f^{-1}(r))$ for any function $f: F\to \mathbb{R}$,
that is $D^f(r)$ denotes the Hausdorff dimension of the function $f$ at level $r$.
We put 
\begin{displaymath}
D_{*}^f(F)=\sup\{d: \lambda\{r : D^f(r,F)\geq{d}\}>0\},
\end{displaymath}
where $\lambda$ denotes the one-dimensional Lebesgue measure.

It is not clear why $D_{*}^{f}(F)$ should have a generic value over $C_{1}^{\aaa}(F)$, thus some extra care is required.
We denote by $\mg_{1,\aaa}(F)$, or by simply $\mg_{1,\aaa}$ the set  of dense $G_{\ddd}$
sets in $C_{1}^{\aaa}(F)$, and we put
\begin{equation}\label{*defDcsaF}
D_{*}(\aaa,F)=\sup_{\cag\in \mg_{1,\aaa}}\inf\{ D_{*}^{f}(F):f\in \cag \}.
\end{equation}

This quantity was the main object of interest in \cite{sier} and \cite{sierc}, to have a short name for it one can call it the $\aaa$-Hölder-thickness of the fractal $F$.

Most of the previous framework can be analogously introduced with the Hausdorff dimension being replaced by the lower or upper box dimension,
as discussed in \cite{buczolich2023box}. While upper box dimension requires some special care and appropriate modifications, which we omit here as we will not
work with upper box dimension in this paper, we can setup the same concepts for lower box dimension
by simply replacing Hausdorff dimension in the original definitions. This is how we define $D_{\underline{B}}^f(r,F),\  D_{\underline{B}*}^f(F),\  D_{\underline{B}*}(\aaa,F)$, as analogues
of $D^f(r,F),\  D_{*}^f(F),\  D_{*}(\aaa,F)$.

In \cite[Theorem~3.3]{sier} (resp. \cite[Theorem~3.2]{buczolich2023box}), we established that $D_{*}(\aaa,F)$ (resp. $D_{\underline{B}*}(\aaa,F)$) is indeed the generic value of
$D_{*}^{f}(F)$ (resp. $D_{\underline{B}*}^f(F)$) over $C_{1}^{\aaa}(F)$, if some technical conditions hold:


\begin{theorem} \label{*thmgenex}
    If   $0< \aaa\leq 1$  and $F\subset\R^p$ is compact, then 
    \begin{itemize}
        \item there is a dense $G_\delta$ subset $\cag$ of $C_1^\alpha(F)$ 
        such that for every $f\in\cag$ we have $D_*^f(F) = D_*(\alpha,F)$.
        \item there is a dense $G_\delta$ subset $\cag$ of $C_1^\alpha(F)$ 
        such that for every $f\in\cag$ we have $D_{\underline{B}*}^f(F) = D_{\underline{B}*}(\aaa,F)$.
    \end{itemize}
\end{theorem}

This result relies on the following lemma (the original statement for Hausdorff dimension is \cite[Lemma~7.1]{sier}, while the extension to lower box dimension is \cite[Lemma~3.1]{buczolich2023box}):


\begin{lemma}\label{*lemdfb}
    Suppose that   $0< \aaa\leq 1$,   $F\subset\R^p$ is compact, $E\subset\R^p$ is open or closed, and $\cau\subset C_1^\alpha(F)$ is open.
    If $\{f_1,f_2,\ldots\}$ is a countable dense subset of $\cau$, then there is a dense $G_\delta$ subset $\cag$ of $\cau$ such that 
    \begin{equation}\label{eqdfb}
    \sup_{f\in\cag} D_*^f(F\cap E) \le \sup_{k\in\N} D_*^{f_k}(F\cap E).
    \end{equation}
    The statement is also valid with $D_*^f, D_*^{f_k}$ being replaced by $D_{\underline{B}*}^f, D_{\underline{B}*}^{f_k}$.
\end{lemma}

We are going to use the mass distribution principle recurringly.


\begin{theorem}\label{thMDP}
    Let $\mu$ be a mass distribution (measure) on $F\subset \R^p$. 
    Suppose that for some 
    $s\geq 0$ there are numbers $c>0$ and $\ddd>0$ such that $\ds \mmm(U)\leq c|U|^{s}$
    for all sets $U$ with $|U|\leq \ddd$. 
    Then $\cah^{s}(F)\geq \mu(F)/c$ and $s\leq \dim F.$
\end{theorem}

We will also use the following approximation result (\cite[Lemma~4.2]{sier}):


\begin{lemma} \label{lipschitzapprox}
Assume that $F$ is compact and $c > 0$ is fixed. 
Then the Lipschitz $c$-Hölder-$\alpha$ functions defined on $F$ form a dense subset $C_c^{\alpha}(F)$.
\end{lemma}


\section{Improved lower estimate for the Sierpiński triangle} \label{sec:sier_triangle_lower}

By its definition the Sierpi\'nski triangle, $\Delta=\bigcap_{n=0}^{\infty}\Delta_n$ where $\Delta_n$ 
is the union of the triangles appearing at the $n$th step of the construction. 
The set of triangles on the $n$th level is $\tau_n$, and $\tau:=\bigcup_{n\in\N\cup\{0\}} \tau_n$.
If $T\in\tau_n$ for some $n$ then we denote by $V(T)$ the set of its vertices. 

In \cite[Section~3]{sierc}, we obtained a lower estimate on $D_*(\alpha, \Delta)$ by defining a function $\kappa: \tau\to (0,1]$,
which we called conductivity, and which interacted well with level sets. The intuition is that if a level set does not intersect 
triangles of high conductivity, then it crosses $\Delta$ through badly conducting triangles. Hence it must intersect many triangles, 
which yields that it has large Hausdorff dimension. The proof of the improved lower estimate presented below follows the same strategy,
however, with a more carefully defined conductivity function $\kappa$. This definition requires some additional notation.

If $n>0$ and $T\in\tau_n$, let $T^-\in\tau_{n-1}$ be such that $T\subset T^-$, and set $T_+:=\{T'\in\tau_{n+1} : T'\subset T\}$ (on the left half of Figure \ref{fig:sierttpl}, $T^{-}$
is the large triangle, $T$ is its lower left subtriangle and $T'$,  one of the elements
of $T_{+}$ is the upper subtriangle of $T$).
Moreover, let $V_{n}$ be the set of the points which are vertices of some $T\in\tau_{n}$, and their union is denoted by $V$. 
We are interested in the Hausdorff dimension of the level sets of a 1-H\"older-$\alpha$ function $f\colon\Delta\to\mathbb{R}$ for $0<\alpha\leq 1$.
We will show that for almost every $r\in f(\Delta)$ the level set $f^{-1}(r)$ intersects ''many'' triangles yielding a high Hausdorff-dimension for $f^{-1}(r)$.

\begin{figure}[h] 
    \centering
    \includegraphics[scale=0.5]{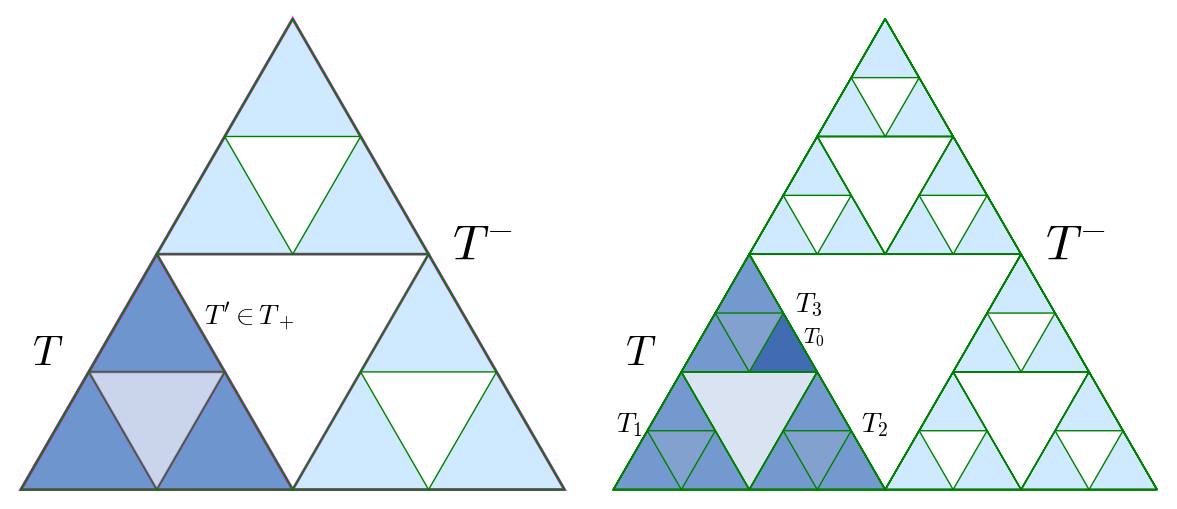} 
    \caption{Left: $T$, $T^{-}$, and the elements of $T_{+}$. Right: definition of $\cas_1,\cas_2,\ldots$. \label{fig:sierttpl}}
\end{figure}

We will define a function $\kappa: \tau\to (0,1]$ and subsets $\cas_1,\cas_2,\ldots$ of $\tau$ inductively.
If $T\in \tau_0\cup\tau_1$, we define $\kappa(T)=1$. 
Let $\cas_1 := \tau_1$.

Let $n\in\N$ and suppose that: 
\begin{itemize}
\item we have already defined $\cas_n$ and $\kappa|_{\cas^*_n}$ where 
$$
\cas_n^* := \{T\in\tau : \exists T'\in\cas_n \text{ such that } T'\subset T\},
$$ 
\item $\Delta\subset\bigcup_{T\in\cas_n} T$,
\item the elements of $\cas_n$ are non-overlapping.
\end{itemize}
(Observe that these conditions are satisfied for $n=1$.)
Take a $T\in\cas_n$.
We denote the elements of $T_+$ by $T_1$, $T_2$ and $T_3$ such that $V(T^-)\cap V(T_1)\neq\emptyset$ (see the right half of Figure \ref{fig:sierttpl}).
 Set $\kappa(T_1):=\kappa(T)$ and $\kappa(T_2):=\kappa(T_3):=\frac12\kappa(T)$.
Let $T_1\in\cas_{n+1}$.
If $T_0\in (T_2)_+\cup(T_3)_+$, then set $\kappa(T_0) := \frac12\kappa(T)$ and let $T_0\in\cas_{n+1}$
(for example on  the right half of Figure \ref{fig:sierttpl} an element of $(T_3)_+$
is shown as $T_0$ and being shaded darker than the rest).
Observe that the induction hypothesis is satisfied for $n+1$ too.

Suppose that $f:\Delta\to\mathbb{R}$ is a 1-H\"older-$\alpha$ function
and $r\notin f(V)$.
 We can define the $n$th approximation of $f^{-1}(r)$ denoted by $G_{n}(r)$ for any $n$ and $r\in f(\Delta)$ as the union of some triangles in $\tau_{n}$. 
More explicitly, $T\in\tau_{n}$ is taken into $G_{n}(r)$ if and only if $T$ has vertices $v$ and $v'$ such that $f(v)< r < f(v')$, that is, $r$ is in the 
interior of the convex hull $\conv (f(V(T)))$. 
The idea is that in this case $f^{-1}(r)$ necessarily intersects $T$.
On Figure \ref{fig:siertpllev} the level set corresponding to $f^{-1}(r)$ is the intersection of
the red curve with $\DDD$. 
On the left and right side of the figure the darkest shaded triangles correspond to
$G_{n}(r)$ and $G_{n+1}(r)$ for a suitably chosen $n$ depending on the choice of the level of the large triangle $T^{-}$.
Now it is easy to check that 
\begin{equation}
 \label{approxlevelunion} \conv (f(V(T)))\subseteq\bigcup_{T'\in T_+}\conv (f(V(T'))),
\end{equation}
hence if $G_{n}(r)$ contains a triangle $T\in\tau_{n}$ then $G_{n+1}(r)$ contains a triangle $T'\in T_+$. 
We introduce the following terminology: we say that $T'\in\tau_{n+k}$ 
is the $r$-descendant of $T\subseteq G_{n}(r)$ if there exists a sequence $T_0=T\supseteq T_1\supseteq\ldots\supseteq T_k=T'$ of triangles such that $T_i\in\tau_{n+i}$ and $T_i\subseteq G_{n+i}(r)$ for $i=0,1,\ldots,k$. 
We denote the set of $r$-descendants of $T$ by $\mathcal{D}_r(T)$.


We need two lemmas:


\begin{lemma}\label{cah_n cap tau_m}
Assume that $T\in \cas_n\cap\tau_m$ and $m,n\in\N$.
Then we have
\begin{displaymath}
\kappa(T) = 2^{n-m}.
\end{displaymath}
\end{lemma}
\begin{proof}
By construction, $n\leq m$.
Suppose $m\in\N$ is fixed. 
We prove by induction on $n$.
If $n=1$, this obviously holds.
Suppose that $n>1$, and we have proved the statement for $n-1$.
Take a $T\in \cas_n\cap\tau_m$. 

If $T^-\in\cas_{n-1}$, then by the definition of $\cas_n$ we have $\kappa(T)=\kappa(T^-)$, hence using the induction hypothesis we obtain
$$
\kappa(T) = \kappa(T^-) = 2^{n-1-(m-1)}=2^{n-m}.
$$

If $T^-\notin \cas_{n-1}$, then $(T^-)^-\in\cas_{n-1}$ and $\kappa(T)=\frac12\kappa((T^-)^-)$, hence using the induction hypothesis we obtain
$$
\kappa(T) = \frac12\kappa\left((T^-)^-\right) = \frac12\cdot2^{n-1-(m-2)} = 2^{n-m}.
$$
\end{proof}




\begin{lemma}\label{d_1 lemma}
If 
$f\colon\Delta\to\R$, $r\in\conv(f(V(\Delta_0)))\setminus f(V)$, $d_1>0$ and 
\begin{equation}\label{kicsi cond}
\sup\{\kappa(T) : T\in\cas_{n} \text{ and } T\cap f^{-1}(r)\neq\emptyset\} < 2^{-nd_1} \text{ for every large enough $n\in\N$},
\end{equation}
then
\begin{displaymath}\label{d_1 becslese}
\dimh(f^{-1}(r))\geq
\frac{d_1}{1+d_1}.
\end{displaymath}
\end{lemma}
\begin{proof}

Fix $f$ and $r$.
We would like to apply  the mass distribution principle, Theorem \ref{thMDP} to $f^{-1}(r)$, hence we define  a Borel probability measure $\mu$ on it. 
Due to Kolmogorov's extension theorem (see for example \cite{Oksendal}, \cite{Taomeas} or \cite{Lamperti})
it suffices to define consistently $\mu(T_r)$ for any $T\in\tau$ where $T_r:=T\cap f^{-1}(r)$. 
Set $\mu((\Delta_0)_r):=1$.
For descendants, we proceed by recursion. 
Notably, if $T$ is an $r$-descendant in $\tau_n$, and $\mu(T_r)$ is already defined, then for an $r$-descendant $T^*\in\tau_{n+1}$ of $T$ we define
\[\mu(T^*_r)=\frac{\mu(T_r)}{\#\{T'\in\tau_{n+1} : T'\text{ is an $r$-descendant of } T\} }.\]
We extend $\mu$ to all Borel subsets $B$ of $\R^2$ by setting $\mu(B):=\mu(B\cap f^{-1}(r))$.


\begin{lemma}\label{mu-lemma}
If $r\in\conv(f(V(\Delta)))$ and $T\in\tau$, then
\begin{equation}\label{mu-lemma kiemelt}
\mu(T) \le \kappa(T).
\end{equation}
\end{lemma}

\begin{proof}Fix $r$.
If $T\in \cas_1^*$, then \eqref{mu-lemma kiemelt} holds obviously.

We proceed by induction.

Suppose that $n>2$ and \eqref{mu-lemma kiemelt} holds if $T\in \cas_{n-1}^*$.

Take a $T^+\in \cas^*_n\setminus \cas_{n-1}^*$ and let $T\in \cas_{n-1}$ be such that $T^+\subset T$.

\begin{figure}[h] 
    \centering
    \includegraphics[scale=0.5]{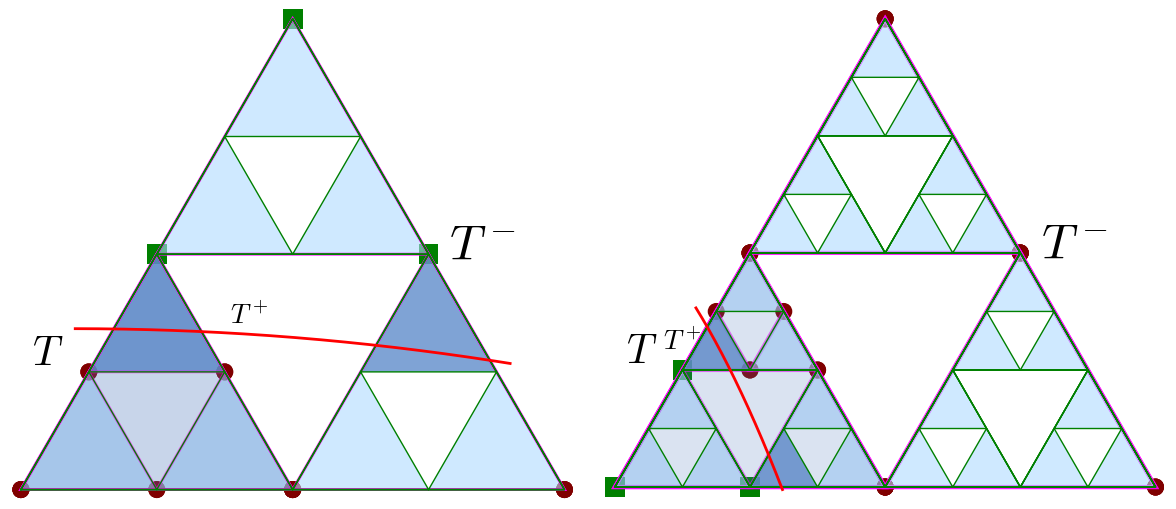} 
    \caption{ Level sets intersecting $T$ \label{fig:siertpllev}}
\end{figure}

If $\kappa(T)=\kappa(T^+)$, then $\mu(T^+)\le\mu(T)\le\kappa(T)=\kappa(T^+)$.

Suppose that $\kappa(T^+)=\frac12\kappa(T)$. In this case either $\diam(T^+) = \frac12\diam(T)$ and $T^+\notin\cas_n$, or $\diam(T^+) = \frac14\diam(T)$ and $T^+\in\cas_n$.
The first case is illustrated on the left side of Figure \ref{fig:siertpllev} while the second case on the right side. In the first case,
 on the left side of the figure at one vertex of $\ds T^{+}$, $f$ takes a value less than $r$. This vertex is marked by a green square. At two other vertices of $\ds T^{+}$, $f$ takes values larger than $r$. 
 These vertices are marked by red dots.
Now one can consider a graph where the vertices are the vertices of the triangles 
in $T_{+}$ and in $(T^{-})_{+}$. These are marked on the left side of Figure \ref{fig:siertpllev} by red dots or green rectangles according to whether $f$ is larger than or smaller than $r$ at these points. The edges of this graph are determined by the sides of the triangles in  $T_{+}$ and $(T^{-})_{+}$.
One can easily verify that if we remove the edges of $T^{+}$ from this graph then we can still connect in the graph one of the red dot vertices of $T^+$
with its green square vertex. This means that we can find an edge with
red dot and green square endpoints. This means that the corresponding triangle,
which is different from $T^{+}$ is also intersected by $f^{-1}(r)$.
Hence $\mu(T^+)\le\frac12\mu(T^-)$. In our example
on  the left side of Figure \ref{fig:siertpllev} the lower right triangle in $(T^{-})_{+}$ is the other triangle intersected by by $f^{-1}(r)$.

The case $\diam(T^+) = \frac14\diam(T)$ and $T^+\in\cas_n$ is illustrated on the right side of Figure \ref{fig:siertpllev}. In this case one is considering the graph determined by the vertices of $T^{+}$, and triangles in  $(T^{-})_{+}$,
$((T^{+})_{-})_{+}$ and
$(((T^{+})_{-})_{-})_{+}$. One can again observe that if we remove from this graph the edges of $T^{+}$ one can still connect any two vertices of $T^{+}$
in the leftover graph. Arguing as before we can again deduce that $\mu(T^+)\le\frac12\mu(T^-)$.

By the definition of $\kappa$, we have $\kappa(T)=\kappa(T^-)$.
Consequently, 
$$
\mu(T^+)\le\frac12\mu(T^-)\le\frac12\kappa(T^-)=\frac12\kappa(T)=\kappa(T^+).
$$
\end{proof}


Assume that $m\in\N$ and $U\subset f^{-1}(r)$ so that $2^{-m}\leq |U| \leq 2^{-(m-1)}$. 
By a simple geometric argument one can show that $U$ intersects at most $C$ triangles in $\tau_{m}\cup\tau_{m-1}$ for some constant $C$ which is
independent of $m$ and $U$. 
(One can consider the triangular lattice formed by triangles with side length $2^{-m}$ and it is easy to see that a set with diameter $2^{-(m-1)}$ can intersect only a limited number of the triangles.) 
Consequently, the number of $r$-descendants of $\Delta_0$ in $\tau_{m}\cup\tau_{m-1}$ intersected by $U$ is bounded by $C$.
Observe that $(\tau_{m}\cup\tau_{m-1})\cap\left(\bigcup_{n\in\N}\cas_n\right)$ covers $\Delta$.

Suppose that $m$ is large enough, $T\in (\tau_{m}\cup\tau_{m-1})\cap\cas_{n}$ for some $n\in\N$ and $T\cap f^{-1}(r)\neq\emptyset$. 
Then by Lemma \ref{cah_n cap tau_m} $\kappa(T)=2^{n-m} $, or
$\kappa(T)=2^{n-(m-1)} $.  
Using \eqref{kicsi cond} as well we obtain
\begin{align*}
2^{n-m} \le \kappa(T) &\le 2^{-d_1n} \\
n-m &\le -d_1 n \\
n(1+d_1) &\le m \\
n &\le \frac{m}{1+d_1}.
\end{align*}
Hence, 
$$\kappa(T) \le 2^{n-m +1} \le 2^{\frac{m}{1+d_1}-m +1 } = 2^{(m-(m+md_1))/(1+d_1) +1} = 2^{-md_1/(1+d_1) +1}.
$$

Thus Lemma \ref{mu-lemma} implies $\mu(T)\le\kappa(T)\le 2^{-md_1/(1+d_1) +1}$.
We obtain that $\mu(U)\le C2^{-md_1/(1+d_1) +1}$.
As $|U|\geq{2^{-m}}$, the mass distribution principle, Theorem \ref{thMDP} tells us that if there exists $C',\ s>0$ independent of $m$ with
\[
C2^{-md_1/(1+d_1)+1}\leq \left(2^{-m}\right)^s C',
\]
then $s\leq \dim_H(f^{-1}(r))$. 
Such a $C'$ exists if and only if
\[s \le\frac{d_1}{1+d_1}.\]
\end{proof}


\begin{theorem}
\label{allexpthm}
Assume that $f:\Delta\to\mathbb{R}$ is a 1-H\"older-$\alpha$ function for some $0<\alpha\leq 1$. 
Then for Lebesgue almost every $r\in f(\Delta)$ we have
\begin{equation}\label{allexpthm_eq}
\dim_H(f^{-1}(r))\geq
\frac{h^{-1}(\alpha)}{1+h^{-1}(\alpha)}
\end{equation}
where $h(d)=\frac{-d\log d-(1-d)\log (1-d)+d\log 6}{(1+d)\log 2}$ and the domain of $h$ is $(0,\frac12]$.

Consequently, $D_{*}(\alpha, \Delta) \geq h^{-1}(\alpha)$.
\end{theorem}

\begin{proof}[Proof of Theorem \ref{allexpthm}]
First we prove that $h^{-1}(\alpha)$ makes sense for every $\alpha\in(0,1]$.
We have
\begin{equation}
\begin{gathered}
h'(d)=\frac{(-\log d - 1 + \log(1-d) +1 +\log 6)(1+d)\log 2}{((1+d)\log 2)^2} \\
- \frac{(-d\log d-(1-d)\log (1-d)+d\log 6)\log2}{((1+d)\log 2)^2},
\end{gathered}
\end{equation}
that is
$$h'(d) = \frac{\log\bigg ( \frac{6}{d}(1-d)^2\bigg ) \log 2}{((1+d)\log 2)^2}.$$
It is immediate that the numerator is positive for $d\in(0,\frac12]$, hence $h'(d)>0$.
Thus $h$ is strictly increasing and it is invertible indeed.
As all terms of the numerator of $h(d)$ tend to $0$ as $d\to0+$, we have that $\lim_{d\to 0+} h(d) = 0$.
Moreover, 
\begin{equation}\label{h(1/2)>1}
h\Big(\frac12\Big)= \frac{-\frac12\log\frac12-\frac12\log\frac12+\frac12\log6}{3/2\cdot\log2} 
= \frac{\log2+\frac12\log6}{3/2\cdot\log2}
> 1,
\end{equation}
hence the domain of $h^{-1}$ contains $(0,1]$.

Since $\DDD$ is compact and connected, $f(\Delta)$ is a closed interval.

Observe that $\bigcup_{T\in\tau_n} \conv(f(V(T)))$ is also a closed interval for every $n\in\N$.
Moreover as $V$ is a countable and dense subset of $\DDD$,
 the set $f(V)$ is countable and dense in $f(\DDD)$.
Suppose that $r\in \inte(f(\DDD))\sm f(V)$.
Then we can find $T\in \tau$ 
 such that $r\in \inte \conv (f(V(T)))$. 
 Thus, it is enough to prove that \eqref{allexpthm_eq} holds for every $T\in \tau$ and for almost every $r\in \conv(f(V(T)))$.
 We will prove it for $T=\Delta_0$ (the other cases can be treated analogously).

By Lemma \ref{d_1 lemma}, it is enough to prove that \eqref{kicsi cond} holds for almost every $r\in\conv(f(V(\Delta_0)))$ and every $d_1 \in (0, h^{-1}(\alpha))$.
For every $n\in\N$ we have
\begin{equation}\label{6-os szamolas.0}
\begin{gathered}
L_{n,d_1,f} := \lambda\left(\left\{r\in\R : \exists T\in\cas_n \text{ for which } f^{-1}(r)\cap T \neq \emptyset \text{ and } \kappa(T)\ge 2^{-nd_1}\right\}  \right) \\
\le \sum_{\substack{T\in\cas_n, \\ \kappa(T)\ge 2^{-nd_1}}} |f(T)|.
\end{gathered}
\end{equation}
For every $n'\in\N$ and $T\in\cas_{n'}$ 
the set $\{T^+\in\cas_{n'+1} : T^+\subset T\}$ 
has $1$ element with conductivity $\kappa(T)$ and $6$ elements with conductivity $\frac12\kappa(T)$,
hence $\cas_n$ has $\binom nk 6^k$ elements with conductivity $2^{-k}$.
Moreover, if $T\in\cas_n$ and $\kappa(T)=2^{-k}$, then $T\in\tau_{n+k}$  by Lemma \ref{cah_n cap tau_m}, hence $|f(T)|\le |T|^\alpha=2^{-(n+k)\alpha}$.
Thus
\begin{equation}\label{6-os szamolas}
\begin{gathered}
L_{n,d_1,f}
\le \sum_{k=0}^{\lfloor nd_1 \rfloor} \binom{n}{k} 6^k \cdot 2^{-(n+k)\alpha}=: M_{n, d_1, \alpha}.
\end{gathered}
\end{equation}

If the series $\sum_{n}L_{n,d_1,f}$ is convergent, then we can apply the Borel--Cantelli lemma to 
deduce that \eqref{kicsi cond} holds for almost every $r\in\conv(f(V(\Delta_0)))$. This can be guaranteed by the convergence of $\sum_{n}M_{n, d_1, \alpha}$. In what
follows, we will precisely determine when this latter series is convergent in terms of $d_1, \alpha$.

Recall that from \eqref{h(1/2)>1} it follows that $d_1<h^{-1}(1)<\frac12$.
As $\binom{n}{k}\le\binom{n}{ \lfloor nd_1 \rfloor }$  and $6^k \cdot 2^{-(n+k)\alpha}$ is monotone increasing in $k$, we obtain
$$
\binom{n}{ \lfloor nd_1 \rfloor } 6^{nd_1} \cdot 2^{-(n+nd_1)\alpha} \le M_{n, d_1, \alpha}
\le (nd_1+1)\binom{n}{ \lfloor nd_1 \rfloor } 6^{nd_1} \cdot 2^{-(n+nd_1)\alpha}.
$$
If we take logarithm and divide by $n$, we find
$$\frac{\log M_{n, d_1, \alpha}}{n} = \frac{\log\binom{n}{\lfloor nd_1 \rfloor}}{n} + \frac{\lfloor nd_1 \rfloor}{n}\log 6 - \alpha\frac{n+\lfloor nd_1 \rfloor}{n}\log 2 + o(1)$$
The floor functions can be removed and the resulting error can be moved into the $o(1)$ term. Moreover, by Stirling's formula, $\log n! = n\log n - n + o(n)$, hence 
\begin{align*}
    \frac{\log\binom{n}{\lfloor nd_1 \rfloor}}{n} &= \Big (\log n - 1\Big ) - \Big (d_1\log(nd_1) - d_1\Big ) \\ &\quad -\Big ((1-d_1)\log(n(1-d_1)) - (1-d_1)\Big ) + o(1).
\end{align*}
where the first three expressions in parenthesis correspond to the factorial terms coming from expanding the binomial coefficient. 
(The floor functions are removed at the price of an $o(1)$ error once again.)
This can be vastly simplified, as most of the terms cancel, and we eventually find
\begin{align*}
\frac{\log M_{n, d_1, \alpha}}{n}
&= -(d_1\log d_1 + (1-d_1)\log(1- d_1)) + d_1\log 6 - \alpha(1+d_1)\log 2 +o(1) \\
&=: c(d_1, \alpha) + o(1).
\end{align*}
Consequently,
$$\sum_{n}M_{n, d_1, \alpha} = \sum_{n} \left(e^{c(d_1, \alpha) + o(1)}\right)^n.$$
But this series converges if and only if $c(d_1, \alpha) < 0$, which can be rearranged to the form $h(d_1) < \alpha$, or equivalently, $d_1 < h^{-1}(\alpha)$. As noted previously, this
yields that \eqref{kicsi cond} indeed holds for almost every $r\in\conv(f(V(\Delta_0)))$ and every $d_1 \in (0, h^{-1}(\alpha))$.
\end{proof}

\begin{remark}
    A potential improvement of this method might come from a different conductivity scheme, which grasps better the geometry of the level sets. While we had some candidates using similar patterns,
    utilizing triangles from more levels, they seemed to yield minor improvements at best at the cost of a lot of technical subtleties, thus we opted to use this construction in the paper. 
    Probably one needs some genuine ideas to come up with a manageable scheme.

    Another improvement might come from the revision of the step where we pass to the series $\sum_{n}M_{n, d_1, \alpha}$, 
    as what we are truly after is the convergence of the series $\sum_{n}L_{n,d_1,f}$. It would be nice to better understand the relationship of these series, as that would probably give
    some insights on the geometry of the level sets.
\end{remark}

We conclude the section by verifying a stronger lower bound on $D_{\underline{B}*}(\alpha, \Delta)$. The proof relies on the observation that the condition of Lemma \ref{d_1 lemma} can be weakened
if we want to draw the same conclusion about the lower box dimension of the corresponding level sets instead of their Hausdorff dimension. Roughly speaking, in the Hausdorff case, in order to guarantee that
a level set lacks "economic" coverings and has a large dimension, upon going to deeper levels in the construction of $\Delta$ one has to require that eventually it does not intersect 
well-conducting triangles at all. However, to guarantee large lower box dimension, 
it suffices that the level set crosses many triangles eventually, which can be already achieved by having an upper bound on the number of intersected well-conducting triangles,
instead of completely ruling out such intersections. This upper bound is hidden in condition \eqref{kicsi cond 2} below.

\begin{lemma}\label{d_1 lemma_B}
If 
$f\colon\Delta\to\R$, $r\in\conv(f(V(\Delta_0)))\setminus f(V)$, $d_1>0$ and
\begin{equation}\label{kicsi cond 2}
\limsup_{n\to\infty}\sum_{\substack{T\in G_n(r) \\ \kappa(T)\ge 2^{-nd_1}}} \kappa(T) < \frac12,
\end{equation}
then 
\begin{equation} \label{d_1 becslese 2}
\ldimb(f^{-1}(r))\geq
d_1.
\end{equation}
\end{lemma}
\begin{proof}
By \eqref{kicsi cond 2}, we can take an $n_0\in\N$ such that for every $n\ge n_0$ 
$$
\sum_{\substack{T\in G_n(r) \\ \kappa(T)\le 2^{-nd_1}}} \kappa(T) > \frac12,
$$
hence $\# G_n(r) \ge 2^{nd_1-1}$. 
Thus, $\ldimb f^{-1}(r) \ge d_1$.
\end{proof}

Essentially, replacing Lemma \ref{d_1 lemma} by Lemma \ref{d_1 lemma_B} in the proof of Theorem \ref{allexpthm} immediately yields a stronger bound on $D_{\underline{B}*}(\alpha, \Delta)$. We present the
core parts of the argument below. (To visualize the improvement, we refer back to Figure \ref{fig:siertestimates}.)

\begin{theorem}
\label{allexpthm_lbdim}
Assume that $f:\Delta\to\mathbb{R}$ is a 1-H\"older-$\alpha$ function for some $0<\alpha\leq 1$. 
Then for Lebesgue almost every $r\in f(\Delta)$ we have
\begin{equation}\label{allexpthm_eq_bdim}
\ldimb(f^{-1}(r))\geq
h_B^{-1}(\alpha)
\end{equation}
where $h_B(d)=\frac{(1-d)\log (1-d)-d\log d-(1-2d)\log(1-2d)+d\log 3}{\log 2}$ and the domain of $h_B$ is $\left(0,\frac{1}{3}\right]$.

Consequently, $D_{\underline{B}*}(\alpha, \Delta) \geq h_B^{-1}(\alpha)$.
\end{theorem}
\begin{proof}
It follows from a simple calculation that $h_B$ is strictly monotone on $\left(0,\frac{1}{3}\right]$ and
its range contains $(0, 1]$, thus the proposed bound indeed makes sense for any $\alpha\in (0, 1]$.

By Lemma \ref{d_1 lemma_B}, it is enough to prove that \eqref{kicsi cond 2} holds for almost every $r\in\conv(f(V(\Delta_0)))$ and every $d_1 \in (0, h_B^{-1}(\alpha))$.
For every $n\in\N$ we have
\begin{equation}
\begin{gathered}
L_{B,n,d_1} := 
\lambda\left(\left\{r\in\R : \sum\nolimits_{\substack{T\in G_n(r),  \kappa(T)\ge 2^{-nd_1}}} \kappa(T) \ge \frac12\right\}  \right) 
\le 2 \sum_{\substack{T\in\tau_n, \\ \kappa(T)\ge 2^{-nd_1}}} |f(T)|\kappa(T).
\end{gathered}
\end{equation}
For every $n,n'\in\N$ and $T\in\tau_n\cap\cas_{n'}$
the set $\{T^+\in\cas_{n'+1} : T^+\subset T\}$ 
has $1$ element with conductivity $\kappa(T)$ (which is in $\tau_{n+1}$) and $6$ elements with conductivity $\frac12\kappa(T)$ (which are in $\tau_{n+2}$),
hence $\tau_n$ has at most $3\binom {n-k}k 6^k$ elements with conductivity $2^{-k}$.
Moreover,  $|f(T)|\le |T|^\alpha=2^{-n\alpha}$.
Thus
\begin{equation}\label{6-os szamolas_B}
\begin{gathered}
L_{B,n,d_1} 
\le 2\cdot 3\sum_{k=0}^{\lfloor nd_1 \rfloor} \binom{n-k}{k} 6^k \cdot 2^{-n\alpha} 2^{-k}
= 6\sum_{k=0}^{\lfloor nd_1 \rfloor} \binom{n-k}{k} 3^k \cdot 2^{-n\alpha} 
=: M_{B,n, d_1, \alpha}.
\end{gathered}
\end{equation}

It is again enough to prove that $\sum_{n}M_{B,n, d_1, \alpha}$ is convergent. 
It can be computed that $h(1/3)>1$, hence $d_1<1/3$.

As $\binom{n-k}{k}\le\binom{n}{ \lfloor nd_1 \rfloor }$, we have
$$
M_{B,n, d_1, \alpha}
\le 6(nd_1+1)\binom{n}{ \lfloor nd_1 \rfloor } 3^{nd_1} \cdot 2^{-n\alpha}.
$$

Analogously to the proof of Theorem \ref{allexpthm} we obtain that
\begin{align*}
\frac{\log M_{B,n, d_1, \alpha}}{n}
&\le (1-d_1)\log(1-d_1)-(d_1\log d_1 + (1-2d_1)\log(1- 2d_1)) \\&\quad + d_1\log 3 - \alpha\log 2 +o(1) \\
&=: c_B(d_1, \alpha) + o(1).
\end{align*}

But this series converges if and only if $c_B(d_1, \alpha) < 0$, which can be rearranged to the form $h_B(d_1) < \alpha$, or equivalently, $d_1 < h_B^{-1}(\alpha)$.
\end{proof}

\section{Improved upper estimate for the Sierpiński triangle} \label{sec:sier_triangle_upper}

Our strategy is similar to the one followed in \cite[Section~4]{sierc}. First one needs to find a
 a suitably chosen H\"older-$\alpha$ function with small level sets in the sense of upper box, and hence of Hausdorff dimension. 
 Then using this function as building block in suitable approximations to functions in a dense set of the corresponding H\"older-space one can obtain the result about generic functions.
The key difference is that in  \cite{sierc} we used functions with "linear" level sets, now we use a tricky construction 
yielding H\"older-$\alpha$ functions with more flexible non-linear level sets. This yields a substantially more complicated construction, but the reward is sufficiently engaging:
the resulting upper bound is very close to our lower bound, both numerically and symbolically, in contrast to our previous upper bound.

The existence of the appropriate building blocks is guaranteed by the following lemma:

\begin{lemma} \label{lemma:upper_estimate}
For every $\alpha\in(0,1)$ and $\eps>0$ there is a H\"older-$\alpha$ function $\fff\colon \Delta\to[0,1]$ such that for a.e. $r\in\R$ we have 
\begin{equation}\label{haromszog_felso_eq}
\dim_H \fff^{-1}(r)\leq \udimb \fff^{-1}(r) \le \frac{h^{-1}(\alpha)}{1+h^{-1}(\alpha)}+\eps, 
\end{equation}
where 
\begin{equation}\label{h definicioja}
h(t)= \frac{-(1-t)\log_2(1-t) - t\log_2(t) + t }{1+t}.
\end{equation}
on the domain $(0, 1/2]$. Moreover, $\varphi$ equals $1$ at one vertex of $\DDD$ and zero at the other two.
\end{lemma}

As the proof is quite complicated, before elaborating it, we discuss its consequences. Lemma \ref{lemma:upper_estimate}
can be used to prove an upper bound on $D_{\underline{B}*}(\alpha, \Delta)$  using precisely the same argument which was used in deducing
\cite[Theorem~4.5]{sierc} from \cite[Lemma~4.4]{sierc}:
using the self-similar structure, a function with small level sets can be used as a building block to create a dense set of functions with small level sets, and then one can apply
on Lemma \ref{*lemdfb}. As an argument of this type will be presented in detail to prove Theorem \ref{thm:cross_small_alpha} from Lemma \ref{a_neq_lemma}, we omit the details here and
simply state the theorem.

\begin{theorem}\label{haromszog_felso}
For every $\alpha\in(0,1)$ we have 
\begin{equation}
D_*(\alpha, \Delta) \le D_{\underline{B}*}(\alpha, \Delta) \le \frac{h^{-1}(\alpha)}{1+h^{-1}(\alpha)}, 
\end{equation}
where 
\begin{equation}
h(t)= \frac{-(1-t)\log_2(1-t) - t\log_2(t) + t }{1+t}
\end{equation}
on the domain $(0, 1/2]$.
\end{theorem}

\begin{remark}
In the following we prove that the upper and the lower bounds for $D_*(\alpha,\Delta)$ given by Theorem \ref{allexpthm} and Theorem \ref{haromszog_felso} are asymptotically equal as $\aaa\to 0+$.
For this it is enough to show that $h_l^{-1}$ (where $h_l$ is the function $h$ in the statement of Theorem \ref{allexpthm}, labeled as ``current lower'' and is shown in dark green on Figure \ref{fig:siertestimates}) is asymptotically equal to $h_u^{-1}$ (where $h_u$ is the function $h$ in the statement of Theorem \ref{haromszog_felso}, labeled as ``current upper'' and is shown in magenta on Figure \ref{fig:siertestimates}),

We have for any $\alpha\in (0,1)$ that
$$\int\limits_{h^{-1}_l(\alpha)}^{h^{-1}_u(\alpha)} h'_l(t) dt 
= h_l(h^{-1}_u(\alpha))-h_l(h^{-1}_l(\alpha))
= h_l(h^{-1}_u(\alpha))-h_u(h^{-1}_u(\alpha)).
$$
Using an estimate of the integral we obtain
$$\left(h^{-1}_u(\alpha)-h^{-1}_l(\alpha)\right) \cdot \min_{t\in [h^{-1}_l(\alpha),h^{-1}_u(\alpha)]} h'_l(t)
\le h_l(h^{-1}_u(\alpha))-h_u(h^{-1}_u(\alpha)).
$$
By rearranging
$$
h^{-1}_u(\alpha)-h^{-1}_l(\alpha) 
\le \left(h_l(h^{-1}_u(\alpha))-h_u(h^{-1}_u(\alpha))\right) \cdot \max_{t\in [h^{-1}_l(\alpha),h^{-1}_u(\alpha)]} \frac1{h'_l(t)}.$$
Therefore, $$ 
\frac{h^{-1}_u(\alpha)-h^{-1}_l(\alpha)}{h_u^{-1}(\alpha)}
\le \frac{h_l(h^{-1}_u(\alpha))-h_u(h^{-1}_u(\alpha))}{h_u^{-1}(\alpha)} \cdot \max_{t\in [h^{-1}_l(\alpha),h^{-1}_u(\alpha)]} \frac1{h'_l(t)}.
$$
Since $|h_u(t)-h_l(t)| = t\log_2 3$ and it can be checked that $\lim_{t\to 0} h'_l(t)= \infty$, the first factor of the right hand-side tends to $\log_2 3$ as $\alpha\to0$ while the second factor tends to $0$, hence the left hand-side tends to $0$ as well. This implies that $h_u^{-1}$ and $h_l^{-1}$ are asymptotically equal indeed. 
\end{remark} 

\begin{proof}[Proof of Lemma \ref{lemma:upper_estimate}]
Similarly to the proof of Theorem \ref{allexpthm}, simple calculation shows that $h$ bijectively maps $(0, 1/2]$ to $(0, 1]$.

The main difficulty of the proof lies in describing the construction of $\varphi$. We will do so in four steps as follows:
\begin{enumerate}
    \item Fixing notation concerning certain subtriangles of $\Delta$.
    \item Heuristic description of $\varphi$, subject to some parameters $k_*, w$ to be fixed.
    \item Defining $\varphi$ in a subset of $\Delta$.
    \item Extending $\varphi$ to $\Delta$.
\end{enumerate}
Following these steps, we will prove in subsequent claims that for well-chosen $k_*, w$, the resulting $\varphi$ has the proposed properties.

\medskip

\noindent {\bf Step 1.}
    Fixing notation concerning subtriangles of $\Delta$ relevant to the construction of the function $\varphi$.

\medskip

\begin{figure}[h] 
\begin{center}
    \includegraphics[width=12.5cm]{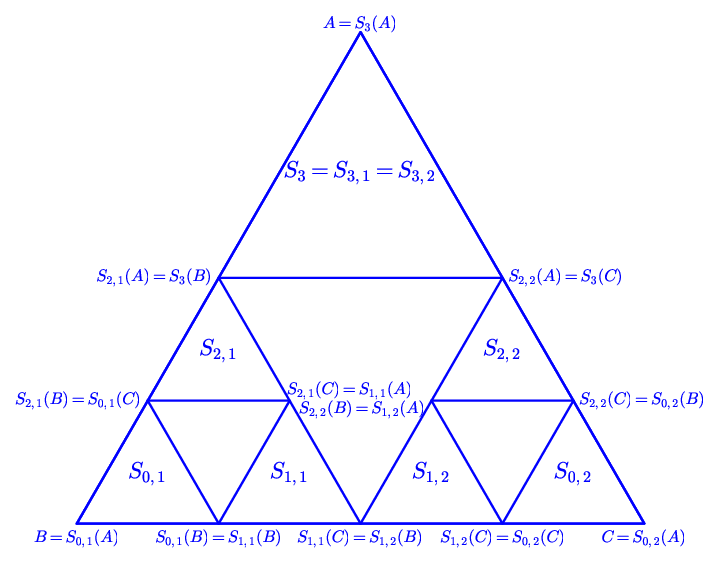}
 \caption{The definition of the similarities.}\label{csucs_def}
 \label{fig_sim}
\end{center}
\end{figure}

Let the vertices of $\Delta$ be denoted by $A$, $B$, and $C$.
Define orientation preserving similarities $S_{0,1},S_{0,2},S_{1,1},S_{1,2},S_{2,1},S_{2,2},S_{3,1},S_{3,2},S_3\colon \R^2\to\R^2$ as in Figure \ref{csucs_def}, 
that is $\ds \DDD$ is mapped onto the triangles labelled by the names of the similarities,
it is also shown which vertices are mapped onto the corresponding vertices of the image triangles.
To be more precise these similarities satisfy the following:
\begin{itemize}
\item $S_{3,1}=S_{3,2}=S_3$, and $S_{3,1},S_{3,2},S_3$ have similarity ratio $\frac12$, the other ones have similarity ratio $\frac14$,
\item each of $S_{1,1},S_{1,2},S_{2,1},S_{2,2},S_3$ can be written as the composition of a uniform scaling transformation and a translation,
\item $S_{0,1}(A)=B$, $S_{0,1}(C)=\frac{A+3B}4$,
\item $S_{0,2}(A)=C$, $S_{0,2}(B)=\frac{A+3C}4$,
\item $S_{1,1}(C)=\frac{B+C}2$, 
\item $S_{1,2}(B)=\frac{B+C}2$, 
\item $S_{2,1}(A)=\frac{A+B}2$, 
\item $S_{2,2}(A)=\frac{A+C}2$, 
\item $S_3(A)=A$. 
\end{itemize}

\begin{figure}[h] 
    \includegraphics[width=\textwidth]{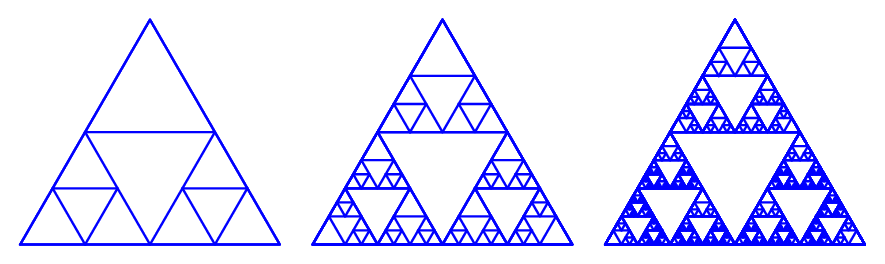}
 \caption{The first three steps of the iteration of the iterated function system $\{S_{0,1},S_{0,2},S_{1,1},S_{1,2},S_{2,1},S_{2,2},S_3\}$.}

\end{figure}

\begin{definition}\label{defdeltaiota}
If $k\in\N$, $\iota=(i_1,\ldots,i_{k})\in\{0,1,2,3\}^{k}$ and 
$\lambda=(l_1,\ldots,l_{k})\in\{1,2\}^{k}$, let 
$$S_{\iota,\lambda} := S_{i_1,l_1}\circ\ldots\circ S_{i_{k},l_{k}},$$
and for every $\iota_1,\ldots,\iota_m\in\{0,1,2,3\}^{k}$ let
$$
\Delta_{\iota_1,\ldots,\iota_m} :=  \bigcup_{\lambda_1,\ldots,\lambda_m\in\{1,2\}^{k}} S_{\iota_1,\lambda_1}\circ\ldots\circ S_{\iota_m,\lambda_m}(\Delta),
$$
see Figure \ref{fig:iota} for some illustrations when $k=3$ and $m=1$. 

If $k,m\in\N$, $\iota_1,\ldots,\iota_m\in\{0,1,2,3\}^k$ 
we  put 
\begin{equation}\label{*frakI}
{\mathfrak I}((\iota_1,\ldots,\iota_m))=(i_{1,1},\ldots,i_{k,1},\ldots,i_{1,m},\ldots,i_{k,m})
\in\{0,1,2,3\}^{km}.
\end{equation}
\end{definition}

\medskip

Observe that
\begin{equation}\label{frakieq}
\Delta_{\iota_1,\ldots,\iota_m} =
\Delta_{{\mathfrak I}(\iota_1,\ldots,\iota_m)}. 
\end{equation}

\medskip

\begin{claim}\label{sides_claim}
 For every $\iota=(i_1,\ldots,i_k)\in\{0,2,3\}^k$ the sets $\Delta_\iota\cap\overline{AB}$ and $\Delta_\iota\cap\overline{AC}$ are line segments (where $\overline{AB}$ and $\overline{AC}$ denote the line segments determined by the points).
\end{claim}
\begin{proof}[Proof of Claim \ref{sides_claim}]
Apply mathematical induction on $k$. 
For $k=1$  by looking at Figure \ref{fig_sim} one can observe that $S_{i,l}(\Delta)$ is not intersecting a side $\overline{AB}$ or
$\overline{AC}$ iff $i=1$. For $i\not=1$ the set $S_{i,1}(\Delta)$ intersects $\overline{AB}$ while $S_{i,2}(\Delta)$ intersects $\overline{AC}$ in one interval. 
 
If $k>1$ and we have proved this statement for $k-1$, then  $\overline{AB}\cap\Delta_{(i_2,\ldots,i_{k})} \neq\emptyset$, hence $\overline{AB}\cap S_{i_1,1}(\Delta_{(i_2,\ldots,i_{k})}) \neq\emptyset$ and $\overline{AC}\cap S_{i_1,2}(\Delta_{(i_2,\ldots,i_{k})}) \neq\emptyset$.
\end{proof}

\begin{figure}[h] 
    \includegraphics[width=0.6\textwidth]{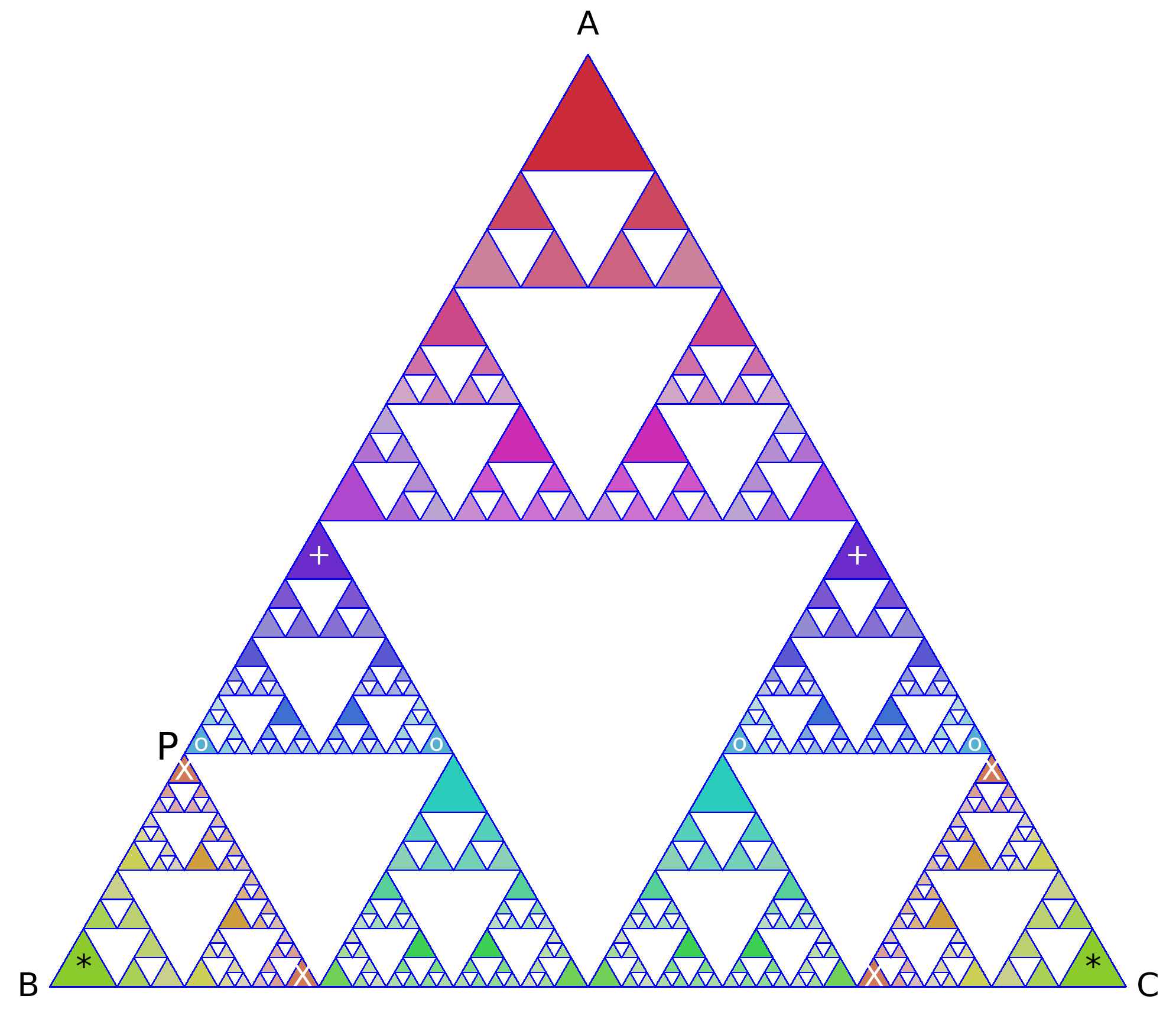}
 \caption{The illustration of
different $\Delta_{\iota}$ sets (same color, same set, but different sets might have only slightly different, hardly distinguishable color), especially the two blue triangles marked by  white $+$s are forming
 $\Delta_{(2,3,3)}$, the two green triangles belonging to $\Delta_{(0,3,3)}$ are marked by black *s,  the four orange/japonica triangles belonging to  $\Delta_{(0,0,3)}$ are marked by tiny white $X$s (they are located close to the interior vertices of the lower left and lower right $1/4$th triangles), the four blue triangles belonging to  $\Delta_{(2,0,3)}$ are marked by tiny white $o$s.}
\label{fig:iota}
\end{figure}

\medskip

\noindent {\bf Step 2.}
    Heuristic description of the function $\varphi$.

\medskip

We will define a function $\varphi\colon\Delta\to\R$ with $\varphi(A)=1$, $\varphi(B)=\varphi(C)=0$, 
which is monotone on line segments $AB$ and $AC$, and it is constant on every $\bigcap_{k=1}^\infty \Delta_{\iota_k}$ where $\iota_k\in\{0,1,2,3\}^k$. 
It will also be constant on $\Delta_\iota$ if a digit of $\iota $ is $1$. The underlying idea is that we want to concentrate the increase of $\varphi$ from 0 to 1 to sets
$\bigcap_{k=1}^\infty \Delta_{\iota_k}$ with small dimension, which happens if most digits of the $\iota_k$s are 3, as it is easy to see heuristically. Notably, these sets yield the narrowest parts of $\Delta$.

To this end, we define the set of digit sequences which are dominated by 3s. 
\begin{definition}\label{defcai}
For $k_*,w\in\N$ to be fixed later we set
\begin{equation}\label{cai def}
\cai := \Big\{(i_1,\ldots,i_{k_*})\subset\{0,2,3\}^{k_*} : \#\big\{k\in\{1,\ldots,k_*\} : i_k\neq 3\big\} \le w\Big\}, 
\end{equation}
such that $\#\cai\ge2$. 
\end{definition}
\begin{remark}
We want to define our function $\fff$ in  a way that $\fff(B)=\fff(C)=0$ hence 
on the segment $BC$ it does not have to change much and therefore it is constant on the triangles
$S_{1,1}(\DDD)$ and $S_{1,2}(\DDD)$ and on their suitable images.  
This is why in the definition
of $\cai$ we use $\{0,2,3\}^{k_*}$, instead of $\{0,1,2,3\}^{k_*}$.
On the other hand, $\DDD$ is "narrow" at $S_{3}(\DDD)$, hence we want to have
most of the increase of $\fff$ on this triangle and on its suitable images,
that is we want to "squeeze" most of the level sets into these regions.
In case of continuous functions this can be done completely, however for
H\"older functions the level sets require more space. This motivates the assumption
 $\big\{k\in\{1,\ldots,k_*\} : i_k\neq 3\big\} \le w$ in the definition of $\cai.$
\end{remark}
Roughly speaking, $\varphi$ will satisfy $|\varphi(\Delta_
\iota)| = (\#\cai)^{-1}|\varphi(\Delta)|$ for every $\iota\in\cai$, while it will be constant on $\Delta_\iota$ for every $\iota\in\{0,1,2,3\}^{k_*}\setminus\cai$, and we will repeat this procedure in the elements of $\tau_{k_*},\tau_{2k_*},\tau_{3k_*},\ldots$ to get a self-affine graph (see Figure \ref{fig:phidef}).

\begin{figure}[h] 
    \includegraphics[width=\textwidth]{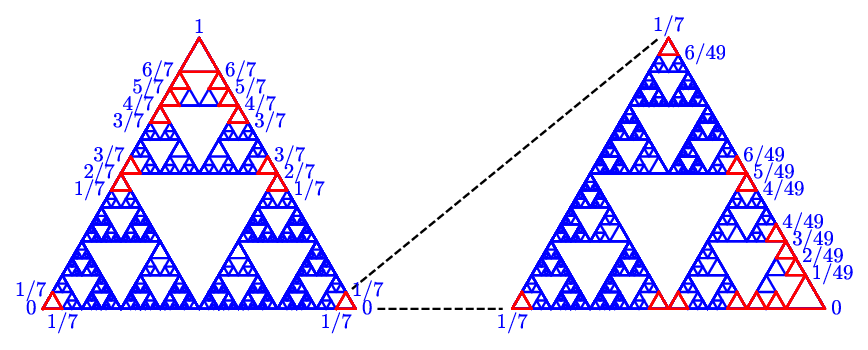}
 \caption{
The elements of of $\{ \Delta_\iota : \iota\in\cai \}$ are represented on the left by the red triangles for $k_*=3$ and $w=1$. 
We can see the values taken by $\varphi$ at the vertices of these triangles ($\varphi$ is constant on each connected blue component).
On the right the next step in one of these triangles is illustrated. Observe that the restriction of $\varphi$ to this small triangle is a self-affine copy of the graph of $\varphi$.}
\label{fig:phidef}
\end{figure}

Let us turn now to the precise definition, suppose that we have already fixed $k_*$, $w$ and $\cai$.

\medskip

\noindent {\bf Step 3.}
    Setting how $\varphi$ grows from $0$ to $1$ as we go from $B$ and $C$ to $A$ on the sides of $\DDD$, and consequently 
    defining $\varphi(x)$ if $x\in\bigcap_{m=1}^\infty \Delta_{\iota_1,\ldots,\iota_m}$ for some $\iota_1,\iota_2,\ldots\in\{0,2,3\}^{k_*}$.
    
\medskip

In order to carry out this step, we define an ordering of $\{\iota\in\{0,2,3\}^{k} : k\in\N\}$. Recall Claim \ref{sides_claim}.
\begin{definition}\label{defk4}
For $k,k'\in\N$, $\iota\in\{0,2,3\}^k$ and $\iota'\in\{0,2,3\}^{k'}$ we say that $\iota<_4\iota'$ if $\overline{AB}\cap\Delta_\iota$ is closer to $B$ than $\overline{AB}\cap\Delta_{\iota'}$. 
Recalling \eqref{*frakI} and \eqref{frakieq} in general
for 
every $m,m'\in\N$ and $\iota_1,\ldots, \iota_m\in\{0,2,3\}^{k}$ and  $\iota_1',\ldots, \iota_{m'}'\in\{0,2,3\}^{k'}$
we put
$$(\iota'_1,\ldots,\iota'_{m'}) <_4 (\iota_1,\ldots,\iota_m)
\text{ iff }
{\mathfrak I}((\iota_1',\ldots,\iota_{m'}'))<_4
{\mathfrak I}((\iota_1,\ldots,\iota_m)).$$
\end{definition}


\begin{claim}\label{casest}
Suppose $\iota_1,\iota_2,\ldots\in\{0,2,3\}^{k_*}$ and $m\in\N$. Then
\begin{equation}\label{cai cdot}
\begin{split}
&\#\{(\iota'_1,\ldots,\iota'_{m+1})\in\cai^{m+1} :  (\iota'_1,\ldots,\iota'_{m+1}) <_4 (\iota_1,\ldots,\iota_{m+1})\} \\
\ge &\#\{(\iota'_1,\ldots,\iota'_{m+1})\in\cai^{m+1} :  (\iota'_1,\ldots,\iota'_m) <_4 (\iota_1,\ldots,\iota_m) \} \\
= &(\#\cai)\cdot \#\{(\iota'_1,\ldots,\iota'_m)\in\cai^m :  (\iota'_1,\ldots,\iota'_m) <_4 (\iota_1,\ldots,\iota_m)\}
=: k'.
\end{split}
\end{equation}
 \end{claim}
 
 \begin{remark}\label{remcasincr}
 Setting $k := k'/\#\cai$, for every $\hat m\in\N$ we have that $\{\iota\in\cai^{m+\hat m} : \iota <_4 (\iota_1,\ldots,\iota_m)\} = k(\#\cai)^{\hat m}$ and $\cai^{m+\hat m}$ have $(\#\cai)^{\hat m}$ elements starting with $(\iota_1,\ldots,\iota_m)$ by the equality in \eqref{cai cdot}.
 \end{remark}

 \begin{proof}[Proof of Claim \ref{casest}]
If $(i_1,\ldots,i_k),(i'_1,\ldots,i'_k)\in \{0,2,3\}^k$ and $(i_1,\ldots,i_k)<_4(i'_1,\ldots,i'_k)$, then for every $i_{k+1},i'_{k+1}\in\{0,2,3\}$ we have 
$(i_1,\ldots,i_k,i_{k+1}) <_4 (i'_1,\ldots,i'_k,i'_{k+1})$\\ as $\Delta_{(i_1,\ldots,i_k,i_{k+1})}\subset\Delta_{(i_1,\ldots,i_k)}$ and 
$\Delta_{(i'_1,\ldots,i'_k,i'_{k+1})}\subset\Delta_{(i'_1,\ldots,i'_k)}$.
\end{proof}



By using Remark \ref{remcasincr} we can define the following function which connects $[0,1]$ and $\{0,2,3\}^{\N}$. 
If $\iota_1,\iota_2,\ldots\in\{0,2,3\}^{k_*}$, then
\begin{equation}\label{phi_cai_def}
\varphi_\cai((\iota_1,\iota_2,\ldots)):=\lim_{m\to\infty}\frac{\#\{(\iota'_1,\ldots,\iota'_m)\in\cai^m :  (\iota'_1,\ldots,\iota'_m) <_4 (\iota_1,\ldots,\iota_m)\}}{(\#\cai)^m}
\end{equation}
(this limit exists, since the fractions form an increasing sequence by \eqref{cai cdot}).

If we have two different sequences $(\iota_1,\iota_2,\ldots),(\iota''_1,\iota''_2,\ldots) \in \left(\{0,2,3\}^{k_*}\right)^\infty$ and 
$$
\bigcap_{m=1}^\infty\Delta_{\iota_1,\ldots,\iota_m} \cap \bigcap_{m=1}^\infty\Delta_{\iota''_1,\ldots,\iota''_m} \neq \emptyset,
$$
then
by induction on $m$ we see that $(\iota_1,\ldots,\iota_m)$ and $(\iota''_1,\ldots,\iota''_m)$ are adjacent with respect to $<_4$, 
which implies that replacing 
$(\iota_1,\ldots,\iota_m)$ with $(\iota''_1,\ldots,\iota''_m)$ in \eqref{phi_cai_def} changes the value of the numerator with at most $1$, hence the limit  in \eqref{phi_cai_def} remains the same.
Thus, if $x\in\bigcap_{m=1}^\infty \Delta_{\iota_1,\ldots,\iota_m}$ for some $\iota_1,\iota_2,\ldots\in\{0,2,3\}^{k_*}$, we can set
\begin{equation}\label{phi_def}
\varphi(x) := \varphi_\cai((\iota_1,\iota_2,\ldots)).
\end{equation}

\medskip

On the set $\{\Delta_{(i_1,\ldots,i_k)} : k\in\N \text{ and } i_1,\ldots,i_k\in\{0,2,3\}\}$, we can define an ordering corresponding to $<_4$ on the indices. With respect to this ordering,
$A$ is in the largest $\Delta_{(i_1,\ldots,i_k)}$ for every $k\in\N$, while $B$ and $C$ are in the smallest one, hence $\varphi(A)=1$ and $\varphi(B)=\varphi(C)=0$ indeed. 

\medskip

\noindent {\bf Step 4.}
    Extending $\fff$ to $\DDD$, namely onto sets of the form
    $\Delta_{(i_1,\ldots,i_k,1)}$ with $i_1,\ldots,i_k\in\{0,2,3\}$ for some $k\in\N\cup\{0\}$.

We put 
\begin{equation}\label{phi_def_2}
\varphi|_{\Delta_{(i_1,\ldots,i_k,1)}} :\equiv \varphi(\Delta_{(i_1,\ldots,i_k,0)}\cap \Delta_{(i_1,\ldots,i_k,2)}).
\end{equation}

We need to show that this definition is correct and does not contradict \eqref{phi_def} (in some vertices $\varphi$ was defined in both \eqref{phi_def} and \eqref{phi_def_2}).

	\medskip

As we intended to focus the increase of $\fff$ on the set considered in Step 3, we aim to deduce that on such connected complementary pieces $\fff$ can be defined to be constant indeed. Thus
further understanding of these sets' relative location will be required.

We start with an example. Suppose that we want to define $\varphi|\Delta_{(1)}:\equiv \varphi(\Delta_{(0)}\cap \Delta_{(2)})$.

We use  Figure \ref{fig:iota} as an illustration. The sequences $(0,0,3)$ and $(2,0,3)$ are adjacent according to $<_{4}$.
The sets $\Delta_{(0,0,3)}$ and $\Delta_{(2,0,3)}$ have one common intersection point $P$ on $\overline{AB}$, which coincides with $\Delta_{(0)}\cap \Delta_{(2)}$ on $\overline{AB}$. This is the common point of a blue triangle marked by a white $o$ and an orange/japonica triangle marked by a white $X$.
Observe the location of those triangles belonging to $\Delta_{(0,0,3)}$ and $\Delta_{(2,0,3)}$ which are not on $\overline{AB}$ or $\overline{AC}$. They have common vertices with the triangles $S_{1,1}(\Delta)$ and $S_{1,2}(\Delta)$ 
(see Figure \ref{csucs_def} as well) and $\fff$ takes the constant value $\fff(P)$ on these triangles.
These two triangles form $\DDD_{1}$. The sequences $(0,0,3)$ and $(2,0,3)$ are initial slices of sequences $\iota^{(0)}$ and $\iota^{(2)}$ such that 
$\Delta_{\iota^{(0)}_{k'}}\cap\Delta_{\iota^{(2)}_{k'}}$ is the same set for every $k'>0$, and its intersection with $\overline{AB}$ is a set with one element,
namely $P$. 

Hence, $\varphi$ can take  a constant value on the triangles $S_{1,1}$ and $S_{1,2}$, and this value is taken at some vertices of the triangles belonging to 
$\Delta_{(0,0,3)}$ and $\Delta_{(2,0,3)}$, marked by $X$s and $o$s.

Now we turn to the proof of the statement about the correctness of the extension of $\varphi$.
Fix a 
 $\Delta_{(i_1,\ldots,i_k,1)}$.
For the sake of simplicity we assume that $(i_1,\ldots,i_k,0) <_4 (i_1,\ldots,i_k,2)$ (the other case can be treated analogously). 
Extending the sequence $(i_1,\ldots,i_k,0)$ successively, we can uniquely define a sequence $\iota^{(0)}\in \{0,2,3\}^{\infty}$ such that
for any $k'>k$ its starting slice $\iota^{(0)}_{k'}$ is the largest with respect to $<_4$ among the elements of $\{0,2,3\}^{k'}$ 
starting with $i_1,\ldots,i_k,0$. 
Similarly, we can uniquely define a sequence $\iota^{(2)}\in \{0,2,3\}^{\infty}$ such that
for any $k'>k$ its starting slice $\iota^{(2)}_{k'}$ is the smallest with respect to $<_4$ among the elements of $\{0,2,3\}^{k'}$ 
starting with $i_1,\ldots,i_k,2$.
Thus, $\Delta_{\iota^{(0)}_{k'}}\cap\Delta_{\iota^{(2)}_{k'}}$ is the same set for every $k'>k$, and its intersection with $\overline{AB}$ is a set with one element,  
namely it is $\overline{AB}\cap\bigcap_{k'=k+1}^\infty \Delta_{\iota^{(0)}_{k'}} = \overline{AB}\cap \bigcap_{k'=k+1}^\infty \Delta_{\iota^{(2)}_{k'}}$.
Moreover,
\begin{equation*}
\begin{gathered}
\Delta_{(i_1,\ldots,i_k,1)}\cap\bigcup_{\iota\in\{0,1,2,3\}^{k+1}\setminus(i_1,\ldots,i_k,1)}\Delta_\iota 
\subset \bigcap_{k'=k+1}^\infty \Delta_{\iota^{(0)}_{k'}} \cup \bigcap_{k'=k+1}^\infty \Delta_{\iota^{(2)}_{k'}}
\end{gathered}
\end{equation*}
(this is obvious for $k=0$, and for larger $k$s it can be obtained using the self-similarity of the IFS). 

\begin{claim}\label{clkonstans}
Let $m_0\in\N$ and $\iota_1,\ldots\iota_{m_0}\in\{0,1,2,3\}^{k_*}$. 
We claim that
\begin{equation}\label{konstans} \text{$\varphi|_{\Delta_{\iota_1,\ldots,\iota_{m_0}}}$ is constant if $\iota_{m'}\notin\cai$ for some $m'\in\{1,\ldots,m_0\}$}.
\end{equation}
\end{claim}

\begin{proof}[Proof of Claim \ref{clkonstans}]
If an $\iota_{m''}$ contains a digit $1$ for an $m''\in\{1,\ldots,m_0\}$, then the claim is obvious from \eqref{phi_def_2}. 

Otherwise (assuming that $m'$ is the smallest integer for which $\iota_{m'}\notin\cai$), the numerator in \eqref{phi_cai_def} is 
$$
(\#\cai)^{m-m'}\cdot \#\{(\iota'_1,\ldots,\iota'_{m'})\in\cai^{m'} :  (\iota'_1,\ldots,\iota'_{m'}) <_4 (\iota_1,\ldots,\iota_{m'})\}
$$ 
for every $m\ge m'$, 
hence $\varphi$ is constant on 
$$
\Delta'_{\iota_1,\ldots,\iota_{m_0}}:=\bigcup_{\iota_{m_0+1},\iota_{m_0+2},\ldots\in\{0,2,3\}^{k_*}}\bigcap_{m''=1}^\infty \Delta_{\iota_1,\ldots,\iota_{m''}}.
$$ 
As 
\begin{equation}\label{lefedi}
\varphi(\Delta_{\iota_1,\ldots,\iota_{m_0}}) = \varphi(\Delta'_{\iota_1,\ldots,\iota_{m_0}})
\end{equation} 
by \eqref{phi_def_2}, we are done. 
\end{proof}

\begin{remark}\label{nem konstans rem}
If $m\in\N$ and $(\iota_1,\ldots\iota_{m})$ is the $k$th element of $\cai^m$, 
then  \eqref{lefedi} and 
Remark \ref{remcasincr}
 imply that $\varphi(\Delta_{\iota_1,\ldots,\iota_m}) = \left[\frac k{(\#\cai)^m},\frac {k+1}{(\#\cai)^m}\right]$.
 \end{remark}

This concludes the construction of $\varphi$.

\medskip

Some important properties of $\varphi$ are summed up by the following claims, the first two concerning with how $k_*, w$ are related to the dimensions of the level sets
and to the regularity of $\varphi$, while the third concerning with the optimization of $k_*, w$ given this knowledge.

\begin{claim}\label{Holder claim}
If $\#\cai \ge 2^{(k_*+w)\alpha}$, then $\fff$ is H\"older-$\alpha$. 
\end{claim}
\begin{claim}\label{dimensions_claim}
	For almost every $r\in \varphi(\Delta)$, we have
	$$\udimb (\varphi^{-1}(r)) \leq \frac{w}{k_*+w}$$
\end{claim}
\begin{claim}\label{optimization_claim}
    We can choose $k_*, w$ so that $\#\cai \ge 2^{(k_*+w)\alpha}$ and 
\begin{equation}\label{optimization_eq}
	\frac{w}{k_*+w} \leq \frac{h^{-1}(\alpha)}{1+h^{-1}(\alpha)}+\eps.
\end{equation}
\end{claim}
Claim \ref{optimization_claim} implies that the level sets are as small as we proposed and $\varphi$ is sufficiently regular, respectively due to Claims \ref{dimensions_claim} and \ref{Holder claim}.
Thus what is left from the proof is verifying these claims.

\begin{proof}[Proof of Claim \ref{Holder claim}]
Before turning to the details of the proof here are some comments about the
function $\fff$, Figure  \ref{fig:phidef} can help to follow them. 
Using the coloring and notation of the left half of this figure we call level-$1$-red triangles
those which belong to $\{ \DDD_{\iota}:\iota\in \cai \}$, and the other ones are called level-$1$-blue.
Recall that
$\fff$ is constant on the level-$1$-blue triangles. The vertices $A,B,C$ of $\DDD$ are vertices of 
level-$1$-red triangles. 
 In case we remove the level-$1$-red triangles $\fff$ is constant on the remaining level-$1$-blue connected components. The distance between these level-$1$-blue components is larger than the height
of the smallest level-$1$-red triangle. Our function $\fff$ was defined in a way that at common vertices of level-$1$-red and level-$1$-blue triangles it is well defined. Now looking at the right half of Figure 
\ref{fig:phidef}
we see that in the interior of a level-$1$-red triangle there are level-$2$-red triangles  of $\{ \DDD_{\iota}:\iota\in \cai^2 \}$.
Some new level-$2$-blue triangles are also 
showing up. 
Due to self similarity any vertex of a level-$1$-red triangle is the vertex of a level-$2$-red triangle. Using this fact, and the properties
of level-$1$-red triangles one can see again that in case we remove the level-$2$-red triangles $\fff$ is constant on the remaining level-$1$ and level-$2$-blue  connected components.
 The distance between these blue components is larger than the height
of the smallest level-$2$-red triangle. By induction analogous statements can be verified
for level-$n$ red and blue triangles.

Take $x,y\in\Delta$ for which $\varphi(x) < \varphi(y)$. 
Let $m\in\N$ be such that 
$$
\varphi(y)-\varphi(x)\in [3(\#\cai)^{-m},3(\#\cai)^{-(m-1)}).
$$
By Remark \ref{nem konstans rem}, the interval  $(\varphi(x),\varphi(y))$ contains $F:=\varphi(\Delta_{\iota_1,\ldots,\iota_m})$ for some $\iota_1,\ldots,\iota_m\in\cai$, that is $x\in\varphi^{-1}([0,\min F))$ and $y\in\varphi^{-1}((\max F, 1])$. 
This remark and Claim \ref{clkonstans} also imply that for every $(\iota'_1,\ldots,\iota'_m)\in\left(\{0,1,2,3\}^{k_*}\right)^m\setminus\{(\iota_1,\ldots,\iota_m)\}$ we have $\varphi(\Delta_{\iota'_1,\ldots,\iota'_m}) \subset[0,\min F]$ or $\varphi(\Delta_{\iota'_1,\ldots,\iota'_m}) \cap [0,\min F] = \emptyset$. 
Hence the set  $\Delta_{\iota_1,\ldots,\iota_m}$ ``separates'' $x$ and $y$ in the sense 
that if we delete it from $\Delta$ then $x$ and $y$ will be in different connected components of the remaining set. 
Fix $\iota_1,\ldots,\iota_m\in\cai$ and
$\lambda_1,\ldots,\lambda_m\in\{1,2\}^{k_*}$. 
Each of $\iota_1,\ldots,\iota_m$ contains at most $w$ digits different from $3$ by \eqref{cai def}. 
As $S_3$ has similarity ratio $\frac12$ and $S_{0,1},S_{0,2},S_{1,1},S_{1,2},S_{2,1},S_{2,2}$ have similarity ratio $\frac14$, 
we have \begin{equation}\label{atm nagy}
\diam(S_{\iota_1,\lambda_1}\circ\ldots\circ S_{\iota_m,\lambda_m}(\Delta)) \ge 2^{-m(k_*-w)-m\cdot2w} = 2^{-m(k_*+w)}.
\end{equation}

Thus 
$$
|x-y|  
\ge \frac{\sqrt3}2 2^{-m(k_*+w)}.
$$ 

Hence, 
\begin{align*}
|\varphi(x)-\varphi(y)|^\alpha 
&\le
3^\alpha(\#\cai)^{\alpha (m-1)} \le \frac{3^\alpha(\#\cai)^{-(m-1)\alpha}}{\big(\frac{\sqrt3}2\big)^\alpha 2^{-m(k_*+w)\alpha}}|x-y|^\alpha \\
& = \left(\frac{\#\cai} {2^{(k_*+w)\alpha}}\right)^{-m\alpha} \cdot \left(\frac{6\#\cai}{\sqrt3}\right)^\alpha \cdot |x-y|^\alpha.
\end{align*}
By the assumption of the claim, this is at most $\left(\frac{6\#\cai}{\sqrt3}\right)^\alpha \cdot |x-y|^\alpha$, hence $\varphi$ is H\"older-$\alpha$.
\end{proof}

\begin{proof}[Proof of Claim \ref{dimensions_claim}]
	The set 
	$$R_c:=\{r\in\R : \varphi^{-1}(r) \text{ contains $\Delta\cap T$ for some $T\in\tau_{n'}$ and $n'\in\N$}\}$$ 
	is obviously countable.
	By the construction of $\varphi$, for every $r\in \varphi(\Delta)\setminus R_c$, the level set can be written in the form $\bigcap_{m=1}^\infty
	\Delta_{\iota^r_1,\ldots,\iota^r_m}$
	for some $\big(\iota^r_{m}\big)_{m=1}^\infty\in\cai^{\infty}$. 
	Thus, for these $r$s we have that
	\begin{equation}
		\begin{split}
			\udimb (\varphi^{-1}(r)) 
			&\underset{\hphantom{\eqref{cai def}}}= \limsup_{n\to\infty} \frac{\log \#\{T\in\tau_{n(k_*+w)} : T\cap \varphi^{-1}(r) \neq\emptyset\}} {\log 2^{n(k_*+w)}} \\
			&\underset{\eqref{cai def}}\le \lim_{n\to\infty} \frac{\log 2^{nw}} {\log 2^{n(k_*+w)}}
			= \lim_{n\to\infty} \frac{nw} {n(k_*+w)}
			= \frac{w}{k_*+w}.
		\end{split}
	\end{equation}
\end{proof}

\begin{proof}[Proof of Claim \ref{optimization_claim}]
By Claim \ref{dimensions_claim}, using the notation $t:=\frac w{k_*}$ we need to minimize $\frac{w}{k_*+w}=\frac t{1+t}$ while not hurting the H\"older property. 
Note that minimizing $\frac t{1+t}$ is equivalent to minimizing $\frac w{k_*} = t$, which is equivalent to minimizing $g(t)$ for any strictly increasing positive function $g$ on $(0,1/2)$, and $h$ is such a function. 
This is how we will get to \eqref{optimization_eq}.
(Of course, the calculations will lead us to the definition of $h$, rather than the definition of $h$ motivating the calculations.)

As  $\#\cai \ge \binom{k_*}w \cdot 2^w$, for the condition on $\#\cai$ in Claim \ref{Holder claim} it suffices to have 
\begin{equation}\label{w kiszamitasa}
\begin{split}
2^{\alpha(k_*+w)} &\le \binom{k_*}w \cdot 2^w = \frac{k_*!}{(k_*-w)!\cdot w!} \cdot 2^w. \\
\end{split}
\end{equation}
By the Stirling formula, 
$$
k^{k+1/2}\cdot e^{-k}\le k! \le e\cdot k^{k+1/2}\cdot e^{-k}$$ 
for every $k\in\N$. Thus
\begin{align*}
\frac{k_*!}{(k_*-w)!\cdot w!} 
&\ge \frac{k_*^{k_*+1/2}\cdot e^{-k_*}} {e\cdot (k_*-w)^{k_*-w+1/2}\cdot e^{-(k_*-w)} \cdot e\cdot  w^{w+1/2}\cdot e^{-w}} \\
&= \frac{k_*^{k_*+1/2}} {e^2\cdot (k_*-w)^{k_*-w+1/2} \cdot w^{w+1/2}}.
\end{align*}
Writing this into \eqref{w kiszamitasa} and taking the base 2 logarithm of both sides, we obtain that it is enough to have
\begin{equation}\label{alpha(...}
\begin{split}
\alpha(k_*+w) 
\le& -\log_2(e^2) + (k_*+1/2)\log_2(k_*) \\
&- (k_*-w+1/2)\log_2(k_*-w) - (w+1/2)\log_2(w) + w \\
\alpha\Big(1+\frac w{k_*}\Big) 
\le& \frac{-\log_2(e^2)}{k_*} + \Big(1+\frac1{2k_*}\Big)\log_2(k_*) \\
&- \Big(1-\frac w{k_*}+\frac1{2k_*}\Big)\log_2(k_*-w) - \Big(\frac w{k_*}+\frac1{2k_*}\Big)\log_2(w) + \frac{w}{k_*}.
\end{split}
\end{equation}
For any $\eps'>0$ we can take a large enough $k_*$ for which  
$$
\frac1{k_*}\left|-\log_2(e^2)+\frac12\log_2(k_*)-\frac12\log_2(k_*-w)-\frac12\log_2(w)\right|\le\eps'. 
$$
Thus, \eqref{alpha(...} follows from
\begin{equation*}
\begin{split}
\alpha\Big(1+\frac w{k_*}\Big) 
\le& \log_2(k_*) - \Big(1-\frac w{k_*}\Big)\log_2(k_*-w) - \Big(\frac w{k_*}\Big)\log_2(w) + \frac{w}{k_*} -\eps'\\
\alpha\Big(1+\frac w{k_*}\Big) 
\le& \log_2(k_*) - \Big(1-\frac w{k_*}\Big)\left(\log_2\Big(1-\frac w{k_*}\Big)+\log_2(k_*)\right) \\
 &- \Big(\frac w{k_*}\Big)\left(\log_2\Big(\frac w{k_*}\Big)+\log_2(k_*)\right) + \frac{w}{k_*} -\eps' \\
\alpha\Big(1+\frac w{k_*}\Big) 
\le& - \Big(1-\frac w{k_*}\Big)\log_2\Big(1-\frac w{k_*}\Big) - \Big(\frac w{k_*}\Big)\log_2\Big(\frac w{k_*}\Big) + \frac{w}{k_*} -\eps'.
\end{split}
\end{equation*}
Using the notation $t=\frac w{k_*}$ this is equivalent to 
\begin{align*}
\alpha(1+t) 
&\le - (1-t)\log_2(1-t) - t\log_2(t) + t -\eps' \nonumber\\
\alpha &\le \frac{-(1-t)\log_2(1-t) - t\log_2(t) + t-\eps'}{1+t} \nonumber,
\end{align*}
which follows from
\begin{equation}\label{alpha+eps}
\alpha+\eps' 
\le \frac{-(1-t)\log_2(1-t) - t\log_2(t) + t}{1+t}.
\end{equation}

Observe that the right hand-side is $h(t)$ from the statement of the lemma. 

If $\eps'$ is small enough, then using the continuity of $h$ we obtain that $h^{-1}(\alpha+\eps')$ is close enough to $h^{-1}(\alpha)$ to satisfy 
\begin{equation*}
\frac{h^{-1}(\alpha+\eps')}{1+h^{-1}(\alpha+\eps')}
< \frac{h^{-1}(\alpha)}{1+h^{-1}(\alpha)}+\eps.
\end{equation*}
We can choose arbitrarily large $k_*$ and $w$ such that $t- h^{-1}(\alpha+\eps')$ is an arbitrarily small positive number, hence $t$ satisfies \eqref{alpha+eps} (since $h$ is strictly increasing) and 
\begin{equation*}
\frac{t}{1+t}
\le \frac{h^{-1}(\alpha)}{1+h^{-1}(\alpha)}+\eps,
\end{equation*}
as claimed.
\end{proof}
This also concludes the proof of Lemma \ref{lemma:upper_estimate}.
\end{proof}


\section{A connected fractal with phase transition} \label{sec:phase_transition}

We start with the construction of the fractal, see Figure \ref{fig:cross_fractal}.

\begin{notation} \label{not:cross_like_fractal}
    Consider the tiling of the unit square $[0, 1]^2$ by closed squares of side length $2^{-m}$. 
    Let $\boxplus_1=\boxplus_1(m)$ 
    be the set remaining from $[0, 1]^2$ after omitting
    certain squares of this partition.
    (We will suppress $m$ in the notation, as it will be fixed during our arguments.)
    Notably, we omit a square if its boundary intersects at least one of the midsegments of $[0, 1]^2$, but does not intersect the
    sides of $[0, 1]^2$. (See Figure \ref{fig:cross_fractal}.)
    The set of remaining squares is denoted by $\mathcal{F}_{1}=\mathcal{F}_{1}(m)$.
    Consider the set of similarities mapping $[0, 1]^2$ 
    to the small squares constituting $\mathcal{F}_{1}$.
    These similarities give rise to an iterated function system. We will denote its attractor by $\boxplus =\boxplus(m)$. 
    The cardinality of $\mathcal{F}_{1}$ will be denoted by $p=p(m)$,
    while the set of small squares constituting the $n$th level of the self-similar construction,
    will be denoted by $\mathcal{F}_{n}=\mathcal{F}_{n}(m)$, 
    and 
    $\mathcal{F} = \bigcap_{n=1}^{\infty}\mathcal{F}_{n}$.   (The notation is extended
    by $\mathcal{F}_{0} = \{[0, 1]^2\}$.)
    The union of all the vertices of squares in $\mathcal{F}_n$ is denoted by $V(\mathcal{F}_n)$, while 
    $V=\bigcup_{n=1}^{\infty}V(\mathcal{F}_n)$.

\end{notation}

\begin{figure}[h] 
    \centering
    \includegraphics[scale=0.3]{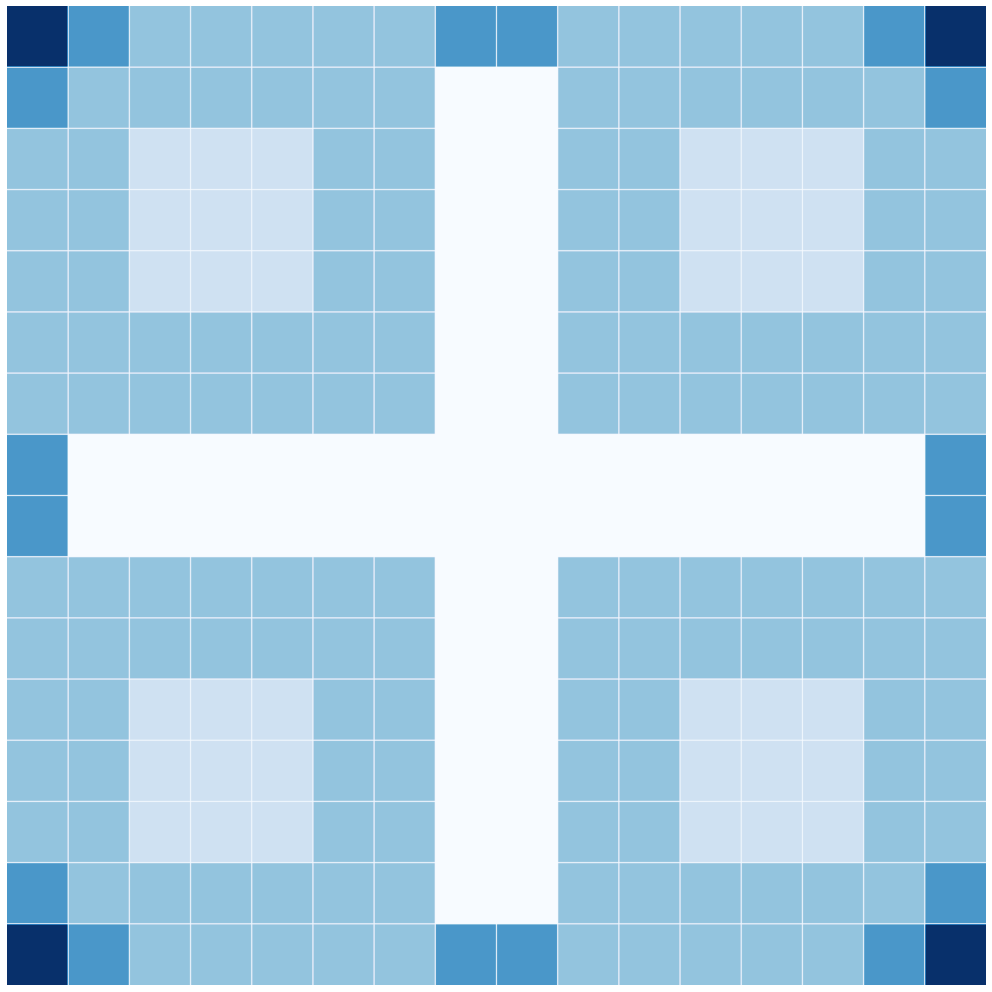}

    \caption{A visualization of $\boxplus_1(4)$, the small squares forming the white cross are omitted from the large square, while
    the others are retained.
    The use of different shades visualizes the proof of Theorem \ref{thm:cross_large_alpha}
    with choice $L=4$:
    gradully lighter shades correspond to higher square type and hence smaller conductivity, introduced in that proof.
    One should note that for fixed $L$, increasing $m$ in $\boxplus_1(m)$ results in the lightest shade, that is Type 4 squares
    dominating the figure, having paramount importance in that proof. \label{fig:cross_fractal}}
\end{figure}

Note that if $Q\in \mathcal{F}_{n}$ for $n>0$, there exists a unique $Q'\in \mathcal{F}_{n-1}$ with $Q\subseteq Q'$. 
We say that $Q$ is a vertical (resp. horizontal) {\it thin square} of $Q'$
if $Q$ falls on a vertical (resp. horizontal) section of $Q'$ which intersects only two subsquares of $Q'$ in $\mathcal{F}_{n}$. 
Otherwise we say that $Q$ is a {\it thick base square}.
Note that the thick base squares of $Q'$ form 4 large squares, to which we refer as {\it thick square}s of $Q'$. 
Now if $Q$ is a thick base square, and $Q\subseteq Q_0$ for 
the thick square $Q_0\subseteq Q'$, we say that $Q$ is of {\it depth} $l$ if its distance from $\partial Q_0$ is at least 
$2^{-nm} (l-1)$. It corresponds to the fact that following a continuous path from such a square, one must traverse at least $l-1$ other thick base squares
to reach the boundary of $Q_0$.

We put $\boxplus_{\text{thin, hor}}$ for the self-similar set which is obtained using only the similarities given by horizontal thin squares. If we rotate it
by $\pi/4$ around the center of $[0,1]^2$, we obtain $\boxplus_{\text{thin, ver}}$.

Now our aim is to prove that for large enough $m$, on $\boxplus=\boxplus(m)$ we encounter the phenomenon of phase transition, 
being the first
connected example we could come up with. To this end, first we have to introduce notions which are more
specific to this fractal, in a similar manner as we did in Section \cite[Sections~4-5]{sierc} in which we carried out
a more detailed analysis of the Sierpiński triangle. The first definition is the analogue of \cite[Definition~4.1]{sierc}:

\begin{definition}
    We say that $f : \boxplus \to \mathbb{R}$ is a {\it piecewise affine function at level $n \in \mathbb{N}$} on
    $\boxplus$ if it is affine on any $Q \in \mathcal{F}_{n}$.

    If a piecewise affine function at level $n \in \mathbb{N}$ on
    $\boxplus$ satisfies the property that for any $Q \in \mathcal{F}_{n}$ one can always find adjacent vertices of $Q$ where
    $f$ takes the
    same value, then we say that $f$ is a {\it standard piecewise affine function at level $n$}. (Note that this property yields
    that restricted to $Q$, either $f(x, y)=h(x)$ or $f(x, y)=h(y)$ for some $h:\mathbb{R}\to\mathbb{R}$ affine function.)
    
    Finally, $f$ is a {\it strongly piecewise affine function} if it is a piecewise affine function at level $n$ for some $n$.
\end{definition}

The following lemma is the analogue of \cite[Lemma~4.2]{sierc}:


\begin{lemma} \label{lemma:piecewise_dense}
    Standard strongly piecewise affine $c^{-}$-Hölder-$\alpha$ functions defined on $\boxplus$ form a dense subset of the $c^{-}$-Hölder-$\alpha$ functions.
\end{lemma}

\begin{proof}
    Due to Lemma \ref{lipschitzapprox}, it suffices to prove that for any  Lipschitz $c^{-}$-Hölder-$\alpha$ function, $f$ we can
    construct standard strongly piecewise affine $c^{-}$-Hölder-$\alpha$ functions arbitrarily close to $f$. 
    Let $M$ be the Lipschitz constant and $c'<c$ be the Hölder constant of $f$.
    
    The construction is the following:
    for some $n$ to be fixed later, let $Q\in\mathcal{F}_{n}$ be arbitrary. 
    If $Q'$ is a thick square of $Q$, such that they share the vertex $v$, 
    let $\tilde{f}|_{Q'\cap \boxplus} \equiv f(v)$. On the thin squares,
    we linearly interpolate from the two closest thick base squares on which $\tilde{f}$ is already defined. As $f$ is $M$-Lipschitz, all these
    linear parts have slope at most $2^{m-1}M$, which yields that $\tilde{f}$ is $\tilde{M}$-Lipschitz for $\tilde{M}=2^{m-1}M$.
    Notice further that the resulting $\tilde{f}$
    is standard piecewise affine at level $n+1$. Moreover, as $f$ is uniformly continuous,
    by
    choosing $n$ large enough $\tilde{f}$ can be arbitrarily close to $f$.
	Hence it suffices to check that it is $c$-Hölder-$\alpha$ as well for large enough $n$. 

    Fix $c''\in (c', c)$. As $\tilde{f}$ is $\tilde{M}$-Lipschitz, for $x,y \in F_n$, we have
    $$|\tilde{f}(x) - \tilde{f}(y)|\leq \tilde{M}|x-y| = \tilde{M}|x-y|^{\alpha}|x-y|^{1-\alpha} \leq c''|x-y|^{\alpha},$$
    if $|x-y|\leq \left(\frac{c''}{\tilde{M}}\right)^\frac{1}{1-\alpha}$. 
    That is, if $x,y$ are close enough, the desired Hölder bound holds, where the sufficient proximity is independent of $n$.
    Thus we only have to handle $x,y$ with $|x-y|$ being separated from 0 by a fixed distance. We proceed as follows (analogously to the end of the proof of \cite[Lemma~4.4]{sier}):
    let $x', y'$ be arbitrary vertices of squares $Q_x, Q_y \in \mathcal{F}_{n}$ so that $x\in Q_x$ and $y\in Q_y$. Then
    $$|\tilde{f}(x) - \tilde{f}(y)| \leq |\tilde{f}(x) - \tilde{f}(x')| + 
    |\tilde{f}(y) - \tilde{f}(y')| + |\tilde{f}(x') - \tilde{f}(y')|,$$
    where the first two terms can be estimated by using the Lipschitz property of $\tilde{f}$, while the last term can be estimated
    by using the Hölder property of $f$, as $f\equiv \tilde{f}$ on vertices of squares in $\mathcal{F}_n$. Notably, we obtain
    $$|\tilde{f}(x) - \tilde{f}(y)| \leq 2\tilde{M}\cdot 2^{-nm} + c'(x-y + 2\cdot 2^{-nm})^{\alpha},$$
    where the right hand side tends to $c'(x-y)^{\alpha}$ as $n$ tends to infinity. Thus $\tilde{f}$ is $c''$-Hölder-$\alpha$ for
    large enough $n$, which concludes the proof.
\end{proof}

We will also construct a function $\Phi$ for which $D_{\underline{B}*}^{\Phi}(\boxplus)\leq \frac{1}{m}$ (and hence $D_{*}^{\Phi}(\boxplus)\leq \frac{1}{m}$) holds. This is going to be
the building block of a dense set of functions with the same $D_{\underline{B}*}, D_{*}$ bounds.

Let $\sum_{j=1}^{\infty}i_j(y)2^{-mj}$ be the base $2^m$ representation of $y$, and let $A$ be the set of numbers in which this
representation uses digits $2^{m-1}-1, 2^{m-1}$ exclusively. Then $A$ is a nowhere dense, perfect set, which in fact coincides with $\boxplus_{\text{thin, hor}}$'s projection to the $x$ axis,
thus it is strongly related to thin parts of $\boxplus$.
For $x\in A$, let
$$\varphi(x)=\sum_{j=1}^{\infty} (i_j(x) - (2^{m-1}-1))2^{-j}.$$
Now $\varphi$ is a monotone function defined on $A$ and takes the same values on endpoints of intervals contiguous to $A$. 
Hence it can be extended
to $[0, 1]$ continuously. Let $\varphi$ denote the extension as well. Clearly $\varphi([0,1])=[0,1]$, and it is straightforward to check that
$\varphi$ is $c_{\varphi}$-Hölder-$\frac{1}{m}$ for some $c_{\varphi}>1$. Consequently, if $\Phi(x, y) = \varphi(x)$, 
then $a\Phi(x, y)+b$ is 
1-Hölder-$\frac{1}{m}$ if $|a|<c_{\varphi}$.


\begin{lemma}\label{a_neq_lemma}
    If $a\neq 0$, then $D_{\underline{B}*}^{a\Phi+b}(\boxplus)= \frac{1}{m}$.
\end{lemma}

\begin{proof}
    It is easy to see that $\varphi$ only admits countably many values on $[0,1]\setminus A$, and $\varphi$ is strictly increasing on $[0,1]\setminus\varphi^{-1}([0,1]\setminus A)$, hence almost every level set is of the form
    $(\{r\}\times [0,1]) \cap \boxplus$ for some $r\in A$. However, in this intersection $\boxplus$ can be replaced by $\boxplus_{\text{thin, hor}}$ for almost every $r \in A$ (the exceptional $r$s being
    the endpoints of contiguous intervals),
    and each vertical section of $\boxplus_{\text{thin, hor}}$ is the same
    self-similar set determined by 2 similarities with ratio
    $2^{-m}$, satisfying the strong separation condition. Consequently, the box dimension of almost every level set is $\frac{1}{m}$.
\end{proof}


\begin{theorem} \label{thm:cross_small_alpha}
$D_{*}(\alpha, \boxplus) = D_{\underline{B}*}(\alpha, \boxplus) = \frac{1}{m}$
for $0<\alpha<\frac{1}{m}$. 
\end{theorem}

\begin{proof}
First we prove that $\frac{1}{m}$ is an upper bound on $D_{\underline{B}*}(\alpha, \boxplus)$, and hence on $D_{*}(\alpha, \boxplus)$ as well.

Due to Lemma \ref{*lemdfb}, it suffices to construct a dense set $\mathcal{T}$ of functions with 
$D_{\underline{B}*}^{f}(\boxplus) \leq \frac{1}{m}$ for $f\in \mathcal{T}$.
Due to Lemma \ref{lemma:piecewise_dense}, it suffices to verify that for any standard strongly piecewise affine $1^{-}$-Hölder-$\alpha$ 
function $f_0$ and $\varepsilon>0$, we can find $f$ with $D_{\underline{B}*}^{f}(\boxplus)\leq \frac{1}{m}$ such that $\|f-f_0\|<\varepsilon$.
Let $M$ be the Lipschitz constant of $f_0$. 
Let $n \in \mathbb{N}$ to be fixed later such that $f_0$ is standard piecewise affine on the $n$th level. We will define $f$ so that
it agrees with $f_0$ on $V(\mathcal{F}_n)$. Moreover, for $Q\in \mathcal{F}_n$ let $\Psi$ be a similarity which maps $Q$ to $[0,1]^2$
so that if $\Phi(\Psi(v))=\Phi(\Psi(v'))$ for some vertices $v, v'\in Q$, then $f_0(v)=f_0(v')$. 
Actually, one can find $a,b\in\mathbb{R}$
with $f_0(v)=a\Phi(\Psi(v)) + b$ for any vertex $v\in Q$, and $|a|<M$. Now define $f$ on $Q\cap \boxplus$ by
$f(x)=a\Phi(\Psi(x)) + b$. Due to the uniform continuity of $f_0$, this construction satisfies $\|f-f_0\|<\varepsilon$ for large enough $n$.
Moreover, if the similarity ratio is large enough, one quickly obtains that $f$ is $1^{-}$-Hölder-$\alpha$ in any $Q\in \mathcal{F}_n$,
which yields the same conclusion globally on $\boxplus$, using a standard argument. (The idea is the following: $f$ is a perturbation of $f_0$ with a particular structure, 
where $f_0$ is a Lipschitz function 
and $f$ is locally $1^{-}$-Hölder-$\alpha$. We need to bound the change of $f$ over large distances. Since any two points can be approximated by auxiliary points in which $f$ coincides with $f_0$,
one can use the Lipschitz bound to effectively bound the change of $f$ and prove that $f$ satisfies the 1-Hölder-$\alpha$ bound over large distances as well. 
The details of the calculation are left to the reader, and can be essentially found
in the proof of Lemma 4.4 in \cite{sier},
in which perturbations of the same type are considered.)
Finally, almost every level set consists of finitely many similar images
of level sets of $\Phi$. That is, $D_{\underline{B}*}^{f}(\boxplus)\leq \frac{1}{m}$ for a dense set of functions.
It concludes the proof $D_{\underline{B}*}(\alpha, \boxplus) \leq \frac{1}{m}$.

For the other direction, notice that $B\times [0,1]\subset \boxplus$, where $B = \Pr_x(\boxplus_{\text{thin, ver}})$ is a strongly separated self-similar set determined by 2 similarities of
ratio $2^{-m}$.
The generic $1$-H\"older-$\alpha$ function $f$ is  non-constant on $\{0\}\times [0,1]$, and as this set is connected, hence $f(\{0\}\times [0,1])$ contains an interval 
$[a-\delta, a+\delta]$. Then due to continuity, $f(\{x\}\times [0,1])$ contains $[a-\frac{\delta}{2}, a+\frac{\delta}{2}]$ for $0\leq x \leq x_0$, $x\in B$.
That is, for any $r\in [a-\delta/2, a+\delta/2]$, the projection of $f^{-1}(r)$ to the first coordinate contains $[0, x_0]\cap B$,
which set has dimension $\frac{1}{m}$ due to the self-similarity of $B$. Consequently, $\dim_H(f^{-1}(r))\geq \frac{1}{m}$ (resp. $\underline{\dim}_B(f^{-1}(r))\geq \frac{1}{m}$) for the generic function
and a positive measure of $r$s. It concludes the proof.
\end{proof}


\begin{theorem} \label{thm:cross_large_alpha}
If $m$ is large enough, there exists $0<\alpha_1<1$, such that $D_*(\alpha, \boxplus)> \frac{1}{m}$ (resp. $D_{\underline{B}*}(\alpha, \boxplus)> \frac{1}{m}$) for $\alpha \in (\alpha_1, 1]$.
\end{theorem}

\begin{proof}
It suffices to prove the statement for $D_*(\alpha, \boxplus)$. The constant $m>2$ is to be fixed at the end of our argument.

The proof will rely on a well-defined conductivity scheme, similarly to what is introduced in Section \ref{sec:sier_triangle_lower}. However, a
simpler construction will be sufficient, making the argument more reminiscent to the proof of \cite[Theorem~3.2]{sierc}. 
First, we introduce similar terminology. If $r\notin f(V)$, which is satisfied for a positive measure of $r$s if $f$ is non-constant, that is, it is satisfied generically,
we define the $n$th approximation of $f^{-1}(r)$ denoted by $G_n(r)$ for any $n$
and $r \in f(\boxplus)$ as the union of some squares in $\mathcal{F}_{n}$. More explicitly, $Q\in \mathcal{F}_{n}$ 
is taken into $G_n(r)$ if and only if $Q$ has vertices
$v, v'$ such that $f(v) < r < f(v')$. The idea is that in this case $f^{-1}(r)$ necessarily intersects $Q$.

We introduce the following terminology: we say that $Q'\in\mathcal{F}_{n+k}$ is the $r$-descendant of $Q\subseteq \mathcal{F}_{n}(r)$ 
if there exists 
a sequence $Q_0=Q\supseteq Q_1\supseteq...\supseteq Q_k=Q'$ of squares such that $Q_i\in\mathcal{F}_{n+i}$ 
and $Q_i\subseteq G_{n+i}(r)$ for $i=0,1,...,k$. 
We denote the set of $r$-descendants of $Q$ by $\mathcal{D}_r(Q)$.

Let $L>0$ to be fixed later.
We will assign conductivity values to each element of $\mathcal{F}$. 
These values are defined by recursion: $\kappa_0([0,1]^2) = 1$. Now if $Q\in \mathcal{F}_{n}$
for $n>0$, there exists a unique $Q'\in \mathcal{F}_{n-1}$ with $Q\subseteq Q'$. 
We separate cases due to the position of $Q$ inside $Q'$. To aid understanding, we refer to Figure \ref{fig:cross_fractal} again.
\begin{itemize}
    \item {\bf Type 1 square.} If $Q$ contains a vertex of $Q'$, let $\kappa_{n}(Q) = \kappa_{n-1}(Q')$. 
    The number of such squares is $t_1 = 4$. 
    \item {\bf Type 2 square.} If $Q$ is the neighbor of a Type 1 square, or if $Q$ is a thin square of $Q'$, 
    let $\kappa_{n}(Q) = \frac{1}{2}\kappa_{n-1}(Q')$.
    The number $t_2$ of such squares is bounded by some constant $K$ independent of $L$ and $m$. 
    \item {\bf Type 3 square.} If $Q$ is not a Type 1 or Type 2 square, and it is of depth $l<L$, 
    let $\kappa_{n}(Q)=\frac{1}{3}\kappa_{n-1}(Q')$. 
    The number $t_3$ of such squares is bounded by $a_L \cdot 2^m$ for some constant $a_L$ independent of $m$.
    \item {\bf Type 4 square.} If $Q$ is not a Type 1 or Type 2 square, and it is of depth $l\geq \frac{L}{2}$, 
    let $\kappa_{n}(Q)=\frac{1}{L}\kappa_{n-1}(Q')$.
    The number $t_4$ of such squares is bounded by $b_L \cdot 2^{2m}$ for some constant $b_L<1$ independent of $m$.
\end{itemize}
We have the following analogue of \cite[Lemma~3.1]{sierc}:


\begin{lemma}
    \label{lemma:cross_approxlevelconductivity}
    Assume that $Q\in G_{n}(r)$ and $k\geq 1$. 
    Then we have
    \begin{displaymath}
    \sum_{Q'\in \tau_{n+k}^l\cap\mathcal{D}_{r}(Q)}\kappa_{n+k}(Q')\geq{\kappa_{n}(Q)}.
    \end{displaymath}
\end{lemma}
\begin{proof}[Proof of Lemma \ref{lemma:cross_approxlevelconductivity}] 
It suffices to consider $k=1$. Now if $G_{n+1}(r)$ contains a Type 1 square or two Type 2 squares, we are done. 
In other cases $G_{n+1}(r)$ must contain a Type 3 or a Type 4 square. 
It is also straightforward to check, that if $G_{n+1}(r)$ contains a Type 3 (resp. Type 4) square, then
it must contain at least 3 (resp. $L$) squares with no lower conductivity, which verifies the statement. \end{proof}

Now assume that $r\in f(\boxplus) \setminus f(V)$. Then $r \in \inte (\conv(Q))$ for some $Q\in\mathcal{F}$. Due to self-similarity properties, 
we might assume $Q=[0, 1]^2$,
that is $G_0(r) = \{ [0, 1]^2 \}$.

For some $1<\beta<\frac{\log 3}{\log 2}$ to be fixed later we would like to estimate the number of squares in $\mathcal{G}_{n}(r)$ with conductivity above $2^{-n\beta}$. 
For $Q_n \in \mathcal{G}_{n}(r)$,
we can take $[0, 1]^2 = Q_0\supseteq Q_1\supseteq...\supseteq Q_n$ such that each square $Q_i$ in the sequence is in $G_i(r)$. 
Assume that $\kappa_m(Q_n) \geq 2^{-n\beta}$.
This yields that out of the squares $Q_1, ..., Q_n$, the number of Type 3 squares is at most $n\beta \frac{\log 2}{\log 3}$,
while the number of Type 4 squares is at most $n\beta \frac{\log 2}{\log L}$. 
(Note that $1 > \beta \frac{\log 2}{\log 3} > \frac{1}{2}$, and for large enough $L$ we have $\beta \frac{\log 2}{\log L} < \frac{1}{2}$.)
Hence the number of squares in $\mathcal{G}_{n}(r)$ satisfying $\kappa_n(Q_n) \geq 2^{-n\beta}$ can be estimated from above by 
$$(t_1+t_2)^n\cdot 
\left(\binom{n}{\lceil{n\left(1-\beta\frac{\log 2}{\log 3}\right)}\rceil} t_3^{\lceil{n\left(1-\beta\frac{\log 2}{\log 3}\right)}\rceil}\right) \cdot 
\left(\binom{n}{\lceil{n\beta\frac{\log 2}{\log L}}\rceil} t_4^{\lceil{n\beta\frac{\log 2}{\log L}}\rceil}\right) = : A_{12}\cdot A_3\cdot A_4 
.$$
(The first factor handles the choice of Type 1 and Type 2 squares,
the next factor the choice of Type 3 squares, while the last one the choice of Type 4 squares.)

With the above notation, we know that the union of $f$-images of such squares has measure at most 
\begin{equation} \label{eq:f_images}
    2^{-nm\alpha}A_{12}A_3A_4.
\end{equation} 
Recalling the bounds on the numbers $t_1, t_2$, we find
\begin{equation} \label{eq:f_images_1}
    A_{12} \leq (K+4)^n,
\end{equation}
    and due to the bound on $t_3, t_4$ and the general inequality $\binom{M}{k} \leq \left(\frac{Me}{k}\right)^k\leq \left(\frac{Me}{k}\right)^M$, 
\begin{displaymath} \label{eq:f_images_2}
    A_3 \leq \left(\frac{e}{1-\beta\frac{\log 2}{\log 3}}\right) ^ n \cdot \left(a_L 2^{m}\right)^{n\left(1-(\beta-\varepsilon)\frac{\log 2}{\log 3}\right)},
\end{displaymath}
\begin{equation} \label{eq:f_images_3}
    A_4 \leq \left(\frac{e}{\beta\frac{\log 2}{\log L}}\right) ^ n \cdot \left(2^{2m}\right)^{n(\beta+\varepsilon)\frac{\log 2}{\log L}}.
\end{equation}
for fixed $\varepsilon > 0$, if $n$ is large enough. (This $\varepsilon$ is used to get rid of ceiling functions.)
Plugging the estimates \eqref{eq:f_images_1}-\eqref{eq:f_images_3} into \eqref{eq:f_images}, 
we find that the union of $f$-images of such squares has measure at most
$$2^{-nm\alpha}\cdot\left(\frac{e}{1-\beta\frac{\log 2}{\log 3}}\right) ^ n \cdot \left(a_L 2^{m}\right)^{n\left(1-(\beta-\varepsilon)\frac{\log 2}{\log 3}\right)}
\cdot \left(\frac{e}{\beta\frac{\log 2}{\log L}}\right) ^ n \cdot \left(2^{2m}\right)^{n(\beta+\varepsilon)\frac{\log 2}{\log L}}
\cdot (K+4)^n=:c^n$$
for large enough $n$.

We claim that for fixed $\beta$, we have $c<1$ if $m, \alpha$ are large enough. Indeed, for fixed $\beta$, if we write $K_1$ 
for all the constant terms above
we can infer
$$c = K_1 \cdot \left(2^{1-(\beta-\epsilon)\frac{\log 2}{\log 3} + 2\cdot(\beta+ \epsilon)\frac{\log 2}{\log L} - \alpha}\right)^m$$
Consequently, for large enough $m$, $c<1$ holds if and only if
$1 - (\beta - \epsilon)\frac{\log 2}{\log 3} + (\beta + \epsilon)\frac{\log 4}{\log L}<\alpha$.
As $\varepsilon>0$ can be set to be arbitrarily small by increasing $n$, 
this inequality can be satisfied if and only if
$1-\beta\left(\frac{\log 2}{\log 3} - \frac{\log 4}{\log L}\right)<\alpha.$
(Note: $L$ can be chosen arbitrarily large by increasing $m$, but $m$ is fixed throughout the construction, so it is of minor importance.)
That is, we have achieved that if 
\begin{equation} \label{eq:beta-alpha}
    1-\beta\left(\frac{\log 2}{\log 3} - \frac{\log 4}{\log L}\right)<\alpha
\end{equation}
holds, 
then for large enough $n$ the $f$-image of squares
in $G_n(r)$ with conductivity exceeding $2^{-n\beta}$ is bounded by $c^n$ for $c<1$. It yields that apart from a null-set of level sets,
the level sets can intersect such squares only for finitely many $n$.
That is, for almost every level set and large enough $n$, $G_n(r)$ contains at least $2^{n\beta}$ elements of $\mathcal{F}_n$.
In other words, $f^{-1}(r)$ intersects at least $2^{n\beta}$ non-overlapping squares of diameter at least $2^{-nm}$.
Now applying  the mass distribution principle, Theorem \ref{thMDP} in the same spirit as in the proof of \cite[Theorem~3.2]{sierc}, we have that 
apart from a null-set, each level set has dimension at least $\frac{\beta}{m}$. 
As generically, there is a positive measure of nonempty level sets,
it yields that $D_*(\alpha, \boxplus)\geq \frac{\beta}{m}$.

This result is a nonempty statement if \eqref{eq:beta-alpha} can hold for some $\alpha<1$ and $\beta>1$, that is
if
\begin{equation} \label{eq:L_large}
\frac{\log 2}{\log 3} > \frac{\log 4}{\log L}.
\end{equation} 
In that case, for 
$\alpha_1 = 1-\frac{\log 3}{\log 2}\left(\frac{\log 2}{\log 3} - \frac{\log 4}{\log L}\right)$ and
$\alpha \in (\alpha_1, 1]$, we can find $\frac{\log 3}{\log 2} > \beta>1$ 
with $1-\beta\left(\frac{\log 2}{\log 3} - \frac{\log 4}{\log L}\right)<\alpha$. 
However, \eqref{eq:L_large} holds if $L$ is large enough. Now if $m$ is chosen large enough, 
the above argument works with this choice of $L$, which concludes the proof.

\end{proof}

\bibliographystyle{amsplain} 
\bibliography{sierimpr} 

\end{document}